\documentclass[12pt]{amsart}
\usepackage[left=3cm, right=3cm, top=3.8cm]{geometry}

\usepackage{microtype}
\usepackage{texproject/styles/local-palatino}

\usepackage{texproject/macros/local-typesetting}
\usepackage{texproject/macros/local-general}
\usepackage{texproject/macros/local-analysis}
\usepackage{texproject/macros/local-tikz}
\usepackage{texproject/macros/local-refs_ams}
\usepackage{texproject/macros/local-theorem_ams}
\usepackage{project-macros}

\newcommand{\titlename}{Geometric and Combinatorial Properties of Self-similar Multifractal Measures}
\newcommand{\shorttitlename}{Self-similar Multifractal Measures}
\newcommand{\docclasses}{Primary 28A80, Secondary 37C45}
\newcommand{\dockeywords}{iterated function system, self-similar, multifractal analysis, weak separation condition}

\subjclass[2020]{\docclasses}

\begin{document}
\title[\shorttitlename]{\titlename}
\author{Alex Rutar}
\thanks{Project supported by NSERC Grants RGPIN-2016-03719 (K. E. Hare) and RGPIN-2019-03930 (K. G. Hare)}
\address{University of Waterloo, 137 University Ave W, Waterloo, ON}
\curraddr{Mathematical Institute, North Haugh, St Andrews, Fife KY16 9SS, Scotland}
\email{alex@rutar.org}

\begin{abstract}
    For any self-similar measure $\mu$ in $\mathbb{R}$, we show that the distribution of $\mu$ is controlled by products of non-negative matrices governed by a finite or countable graph depending only on the IFS.
    This generalizes the net interval construction of Feng from the equicontractive finite type case.
    When the measure satisfies the weak separation condition, we prove that this directed graph has a unique attractor.
    This allows us to verify the multifractal formalism for restrictions of $\mu$ to certain compact subsets of $\mathbb{R}$, determined by the directed graph.
    When the measure satisfies the generalized finite type condition with respect to an open interval, the directed graph is finite and we prove that if the multifractal formalism fails at some $q\in\mathbb{R}$, there must be a cycle with no vertices in the attractor.
    As a direct application, we verify the complete multifractal formalism for an uncountable family of IFSs with exact overlaps and without logarithmically commensurable contraction ratios.
\end{abstract}

\maketitle
\tableofcontents

\section{Introduction}
Self-similar measures in $\R$ are perhaps the simplest examples of measures which exhibit complex local structure.
These measures are associated with finite sets of similarity maps in $\R$.
To be precise, by an \emph{iterated function system of similarities} (IFS) we mean a finite set of maps $\{S_i\}_{i\in\mathcal{I}}$ where each $S_i(x)=r_i x+d_i$ and $0<|r_i|<1$.
The \emph{attractor}, or \emph{self-similar set}, of this system is the unique compact set $K$ satisfying $\bigcup_{i\in\mathcal{I}}S_i(K)=K$.
Given a probability vector $\bm{p}=(p_i)_{i\in\mathcal{I}}$ where each $p_i>0$ and $\sum_i p_i=1$, the associated \emph{self-similar measure} is the unique Borel probability measure satisfying
\begin{equation*}
    \mus(E)=\sum_{i\in\mathcal{I}}p_i\mus\circ S_i^{-1}(E)
\end{equation*}
for any Borel set $E\subseteq\R$.
For a more through discussion of the background and basic properties of self-similar sets and measures, we refer the reader to Falconer's book~\cite{fal1997}.

In order to understand the general structure of the measure $\mus$ or the self-similar set $K$, one often considers basic dimensional quantities such as the Hausdorff dimension $\dimH K$ and analogous statements for measures, or other notions of dimension.
Computing these values can be highly non-trivial for general iterated function systems of similarities and there is significant literature on this matter (see, for example, \cite{bg1992,fh2009,fhor2015,hoc2014,jr2021,ln2007,nw2001,sch1994}).
In this paper, we focus on a more fine-grained notion of dimension known as the local dimension.
Given a point $x\in K=\supp\mus$, the \emph{local dimension} is given by
\begin{equation*}
    \diml\mus(x)=\lim_{t\to 0}\frac{\log \mus(B(x,t))}{\log t},
\end{equation*}
when the limit exists.
From the perspective of multifractal analysis, one is interested in determining geometric properties of the sets $K(\alpha)\coloneqq \{x\in K:\diml\mus(x)=\alpha\}$.
On the other hand, the \emph{$L^q$-spectrum} of $\mus$ is given by
\begin{equation*}
    \tau(\mus,q)=\tau(q)\coloneqq  \liminf_{t\to 0}\frac{\log \sup\sum_i\mus(B(x_i,t))^q}{\log t}
\end{equation*}
for each $q\in\R$, where the supremum is over disjoint families of closed balls with centres $x_i\in K$.

An important objective of multifractal analysis is to understand the relationship between the $L^q$-spectrum of the measure $\mus$, and the \emph{dimension spectrum} $\dimH K(\alpha)$.
A heuristic relationship between $\tau(q)$ and $\dimH K(\alpha)$, known as the multifractal formalism, was introduced by Halsey \textit{et al.}~\cite{hjk+1986}.
The \emph{multifractal formalism} states, roughly speaking, that the dimension spectrum can be computed as the concave conjugate of $\tau(q)$, i.e.
\begin{equation*}
    \dimH K(\alpha) =\tau^*(\alpha)\coloneqq \inf_{q\in\R}\{q\alpha-\tau(q)\}
\end{equation*}
for any $\alpha$ in the domain of $\tau^*(\alpha)$; see \cref{d:cmf} for a complete definition in our setting.
This concave conjugate relationship has been studied by many authors (see, for example, \cite{cm1992,fen2003a,fen2009,fl2009,flw2005,hjk+1986,lau1995,ln1999,pat1997,pw1997,shm2005}).
As a particularly elegant example, it has been verified in general for iterated function systems satisfying the strong separation condition ($S_i(K)\cap S_j(K)\neq\emptyset$ if and only if $i=j$)~\cite{cm1992}.
This separation requirement has been relaxed to the open set condition~\cite{hut1981} and the concave conjugate relationship has been verified~\cite{ap1996,pat1997,pw1997}.
In both cases, $\tau(q)$ is differentiable for all $q\in\R$ and is determined uniquely by the implicit formula $\sum_{i\in\mathcal{I}}p_i^q r_i^{-\tau(q)}=1$.

However, when the open set condition fails, outside specialized analysis of some families of examples (for example, Bernoulli convolutions associated with the unique positive root of the polynomial $x^k-x^{k-1}-\cdots-x-1$ \cite{fen2005}), there has been much less progress on verifying the multifractal formalism at all $q\in\R$.
For $q\geq 0$, the function $x\mapsto x^q$ is non-decreasing so the summation in the definition of $\tau(q)$ is dominated by closed balls with large measure.
On the other hand, for $q<0$, the summation is dominated by closed balls of small measure.
Generally speaking, understanding the multifractal analysis of measures when $q<0$ is substantially more challenging than the case $q\geq 0$.
Gaining more information about this case is our focus in this document.

\subsection{The weak separation condition}
Notably, neither the strong separation condition nor the open set condition allows for the existence of exact overlaps.
We introduce some notation: let $\mathcal{I}^*$ denote the set of all finite words on $\mathcal{I}$.
For $\sigma=(i_1,\ldots,i_n)\in\mathcal{I}^*$, write $S_\sigma=S_{i_1}\circ\cdots\circ S_{i_n}$, $r_\sigma=r_{i_1}\cdots r_{i_n}$ and, if $n\geq 1$, $\sigma^-=(i_1,\ldots,i_{n-1})$.
By \defn{exact overlaps} we mean the existence of words $\sigma\neq \tau\in\mathcal{I}^*$ such that $S_\sigma=S_\tau$.
To study examples allowing exact overlaps while still maintaining separation of non-overlapping words, Lau and Ngai introduced the weak separation condition and studied basic conditions under which the multifractal formalism holds~\cite{ln1999}.
For any $t>0$ and Borel set $E\subseteq\R$, define
\begin{equation*}
    \Lambda_t(E) = \{\sigma\in\mathcal{I}^*:r_\sigma< t\leq r_{\sigma^-},S_\sigma(K)\cap E\neq\emptyset\}.
\end{equation*}
Then the \emph{weak separation condition} is equivalent to requiring that
\begin{equation}\label{e:int-max}
    \sup_{x\in\R,t>0}\#\{S_\sigma:\sigma\in\Lambda_t(U(x,t))\}<\infty
\end{equation}
where $\# X$ denotes the cardinality of a set $X$ and $U(x,t)$ is the open ball about $x$ with radius $t$.
Note that the definition only considers functions $S_\sigma$ rather than the words $\sigma$ so as to allow exact overlaps.
To see an equivalent formulation with respect to exact overlaps or the equivalence with the original definition of Lau and Ngai, see~\cite[Thm.~1]{zer1996}.

Under the weak separation condition, verification of the multifractal formalism is subtle.
One of the earliest examples of exceptional behaviour is with respect to self-similar measures of the system of Bernoulli convolutions $\{x\mapsto \rho x,x\mapsto \rho x+(1-\rho)\}$ where the contraction ratio $\rho$ is the reciprocal of the golden mean.
In this case, the $L^q$-spectrum $\tau(q)$ has a \emph{phase transition}, or a point where $\tau(q)$ is not differentiable.
Nevertheless, the multifractal formalism still holds and $\tau(q)$ is analytic for other values of $q$~\cite{fen2005}.
Another example of exceptional behaviour is the $3$-fold convolution of the uniform Cantor measure.
In this case, it was observed that the set of attainable local dimensions is not an interval and the multifractal formalism fails~\cite{hl2001}.
The problem here is, in some sense, that the measure $\mus$ is too small at certain points in $K$.
This measure, and other related measures, were studied in detail~\cite{flw2005,hhs2021,lw2005,shm2005} and a modified multifractal formalism was proven therein.
In these cases, the failure occurs at some point $q<0$.

In an important paper, Feng and Lau~\cite{fl2009} obtain deep results about the multifractal formalism under the weak separation condition.
Using a subtle Moran construction~\cite{flw2002}, they prove that the multifractal formalism holds for any value $q\geq 0$, and for $q<0$, they give a modified multifractal formalism by considering suitable restrictions to an open ball $U_0$ which attains the supremum in the definition of the weak separation condition \cref{e:int-max}.
Unfortunately, this result does not directly give information on the validity of the multifractal formalism for values $q<0$.
In some sense, the restriction avoids the breakdown of the multifractal formalism by avoiding points in $K$ where the measure is too small.

To extend this perspective, we develop some new ideas.
Even in regions where the overlap is not dense (i.e. away from any maximal open ball $U_0$), through a general graph construction, we will show that the measure may be ``combinatorially linked'' to regions with high density where the multifractal formalism holds.
For example, consider the IFS given by the maps
\begin{align}\label{e:exifs}
    S_1(x)&=\rho x & S_2(x)&=r x+\rho(1-r) & S_3(x)&=r x+1-r
\end{align}
where $\rho>0$, $r>0$ satisfy $\rho+2r-\rho r\leq 1$.
This IFS was first studied by Lau and Wang~\cite{lw2004} and satisfies the weak separation condition.
In \cref{sss:valid-open}, we show that the maximal open sets $U_0$ can never contain the point $1$ in the self-similar set, which is a phenomenon similar to the situation of the Cantor convolution.
Despite this, we can prove (as a consequence of our more general results) that the multifractal formalism still holds for the measure $\mus$, without restriction to a subset and with any probabilities.
Our main goal in this paper is to provide a new, natural perspective for understanding the failure of the multifractal formalism, and to provide combinatorial conditions under which the multifractal formalism holds or in which one might suspect that fails.

Our starting point is the net interval construction of Feng~\cite{fen2003}.
In that document, for iterated function systems of the form $\{x\mapsto rx+d_i\}_{i\in\mathcal{I}}$ with $0<r<1$ satisfying a combinatorial overlap condition known as the finite type condition~\cite{nw2001}, he obtains formulas for the values of $\mus(\Delta)$ on families of intervals $\mathcal{F}_n$ as products of non-negative matrices.
He then uses properties of matrix products to verify differentiability of the $L^q$-spectrum (and thus the multifractal formalism by the prior work of Lau and Ngai~\cite{ln1999}) for values $q>0$.
Using some different perspectives but with the same underlying approach, he proves a modified multifractal formalism for values of $q<0$~\cite{fen2009}.

In recent work, following the techniques of Feng and operating in the same setting, Hare, Hare, and various collaborators~\cite{hhm2016,hhn2018} define a finite graph called the transition graph corresponding to the IFS.
Then they determine that the set of local dimensions at special points in $K$ called interior essential points form a closed interval, and show that the failure for the set of local dimensions to be a closed interval is determined by the existence of certain combinatorial structures in the transition graph called non-essential loop classes.

However, as observed by Testud~\cite{tes2006}, when the IFS does not have a common contraction ratio or a similar property (for example, $\log r_i/\log r_j\in\Q$ for all $i,j$~\cite{hhs2018}), one cannot apply Feng's net interval construction in a natural way.

\subsection{Summary of main results}
Our first contribution is a generalization of the net interval construction to apply to any IFS of similarities.
We determine that the distribution of $\mus$ on certain intervals which we call \emph{net intervals} is determined by a local overlap structure which we call the \emph{neighbour set} of the net interval (see~\cite{hhr2021} for the first appearance of this construction).
Our first key observations, \cref{l:nbdet} and \cref{t:ttype}, are that the neighbour set completely determines the local geometry of the attractor $K$ and the distribution of the measure $\mus$ (up to fixed constants of comparability).
This allows us in \cref{ss:tgraph} to construct a countable directed graph which we call the \emph{transition graph} of the IFS, where the vertices are the distinct neighbour sets.
Then in \cref{ss:tg-im}, we associate to each edge of the transition graph a non-negative matrices called a \emph{transition matrix} such that the distribution of $\mus$ on net intervals is given by products of these non-negative matrices.
Since we do not make any assumptions on the contraction ratios, we introduce two simple but important ideas: the notion of the \emph{transition generation} (\cref{d:tg}), and the notion of the \emph{length of an edge} (\cref{d:e-len}).
These definitions resolve the issues with the original net interval construction recognized above.

In \cref{s:ifswsc}, we turn our attention to the IFSs satisfying the weak separation condition.
In particular, we prove the existence of a relatively open subset $K_{\ess}\subseteq K$ called the \emph{set of interior essential points}, and a corresponding subgraph of the transition graph called the \emph{essential class} on which the self-similar measure has certain important regularity properties (\cref{l:bmeas}).
We call a net interval \emph{essential} if its neighbour set is a vertex in the essential class.
We determine that the set of interior essential points is large in two different senses:
\begin{theorem}\label{t:ess-large}
    Let $\{S_i\}_{i\in\mathcal{I}}$ be an IFS satisfying the weak separation condition.
    \begin{enumerate}[nl,r]
        \item If $U_0$ is any open set which attains the maximality in \cref{e:int-max}, then $K\cap U_0$ is contained in a finite union of essential net intervals.
            In particular, $K\cap U_0\subseteq K_{\ess}$.
        \item If $\mus$ is any associated self-similar measure, then $\mus(K\setminus K_{\ess})=0$.
    \end{enumerate}
\end{theorem}
See \cref{p:mx-wsc} and \cref{t:ess-meas} for proofs of these facts.

We also obtain dimensional results at certain points in $K$ called \emph{periodic points}, an idea introduced by Hare, Hare, and Matthews.
In \cref{p:period}, we prove that an elegant formula holds for the local dimensions at such points, and in \cref{t:per-dens} we show that the sets of local dimensions at periodic points are dense in the sets of upper and lower local dimensions at points in $K_{\ess}$.
This generalizes a pre-existing result~\cite[Cor.~3.15]{hhn2018} to the weak separation case.

We then focus on understanding the multifractal formalism from the perspective of the essential class.
We introduce the notion of \defn{weak regularity} in \cref{d:weak-reg}.
Our main result in this section is the following (see \cref{t:multi-wsc} for a complete statement and proof):
\begin{theorem}\label{ti:multi-wsc}
    Let $\{S_i\}_{i\in\mathcal{I}}$ be an IFS satisfying the weak separation condition and let $\mus$ be an associated self-similar measure.
    Let $E=\Delta_1\cup\cdots\cup\Delta_n$ be a finite union of essential net intervals such that $E\cap K$ is weakly regular.
    Then $\nu=\mus|_E$ satisfies the multifractal formalism and
    \begin{equation}
        \{\diml\nu(x):x\in\supp \nu\} = \{\diml\mus(x):x\in K_{\ess}\}.
    \end{equation}
    Moreover, the values of $\tau(\nu,q)$ do not depend on the choice of $\Delta_1,\ldots,\Delta_n$ and for $q\geq 0$, $\tau(\mus,q)=\tau(\nu,q)$.
\end{theorem}

Our verification of this modified multifractal formalism begins with~\cite[Thm~1.2]{fl2009}, but then uses the matrix product structure of the transition graph to move the weight of the measure from the sets $U_0$ to any net interval in the essential class.
We note some minor improvements: rather than considering restrictions of the $L^q$-spectrum to an open set, we obtain the results as a restriction to a compact subset $\Delta_1\cup\cdots\cup\Delta_n$, where this subset can strictly contain a neighbourhood of any open set $U_0$ attaining the maximum in \cref{e:int-max} (combine \cref{t:ess-large} and \cref{l:er-extension}).
This boundary regularity condition is discussed in detail in \cref{ss:weak-reg}.

In fact, our matrix product structure provides a more general perspective for understanding the quasi-product property of Feng and Lau~\cite{fl2009}; a natural analogue holds in our setting where their set $\Omega$ is replaced by a set of net intervals which have the neighbour of a fixed essential net interval.
As a result, a more direct proof of \cref{ti:multi-wsc} is possible.
However, many details of this proof overlap with the approach of Feng and Lau, so we do not include this approach.

Combining this result with \cref{t:ess-large}, we prove the following modified multifractal formalism for any IFS satisfying the weak separation condition:
\begin{corollary}
    Let $\{S_i\}_{i\in\mathcal{I}}$ be an IFS satisfying the weak separation condition with associated self-similar measure $\mus$.
    Then there exists a sequence of compact sets $(K_m)_{m=1}^\infty$ with $K_m\subseteq K_{m+1}\subseteq K$ for each $m\in\N$ such that
    \begin{enumerate}[nl,r]
        \item $\lim_{m\to\infty}\mus(K_m)=1$,
        \item each $\mu_m\coloneqq \mus|_{K_m}$ satisfies the multifractal formalism, and
        \item $\tau(\mu_m,q)$ and $D(\mu_m)$ do not depend on the index $m$.
    \end{enumerate}
\end{corollary}
We note the similarity of this result to a result of Feng~\cite[Thm.~1.2]{fen2009}, which follows from general results about the multifractal formalism of certain matrix-valued functions satisfying an irreducibility condition.
However, the techniques used therein only apply naturally in the finite type case for IFSs of the form $\{x\mapsto rx+d_i\}_{i\in\mathcal{I}}$.

We also obtain the following important corollary:
\begin{corollary}\label{c:mfcor}
    Let $\{S_i\}_{i\in\mathcal{I}}$ be an IFS satisfying the weak separation condition with transition graph $\mathcal{G}$.
    Suppose there is a bound on the maximum length of a path with no vertices in the essential class.
    Then any associated measure $\mus$ satisfies the multifractal formalism.
\end{corollary}
In particular, suppose $\mathcal{G}$ is finite.
In this situation, the only mechanism for the failure of the multifractal formalism is the existence of a cycle (a path in the transition graph which begins and ends at the same vertex) which is not contained in the essential class.
This gives a combinatorial condition which guarantees that the multifractal formalism holds.
In this situation, it is possible to write a finite algorithm to determine whether such a cycle exists.

In particular, in \cref{t:rr-multif}, we apply this to the family of IFS defined in \cref{e:exifs}:
\begin{corollary}
    Let $\{S_i\}_{i=1}^3$ be the IFS defined in \cref{e:exifs}.
    Then for any probability weights $\bm{p}=(p_i)_{i=1}^3$, the associated self-similar measure $\mus$ satisfies the complete multifractal formalism.
\end{corollary}
To the best knowledge of the author, this is the first example of an IFS with exact overlaps and without logarithmically commensurable contraction ratios for which the complete multifractal formalism is proven to hold.
Understanding failure of the multifractal formalism is based critically on understanding the properties of cycles in the transition graph outside the essential class.

By combining our results with the work of Deng and Ngai \cite{dn2017}, we can also gain information about differentiability of the $L^q$-spectrum.
In a slightly specialized case, \cite[Thm.~1.2]{dn2017} states that, for probabilities $p_2>p_3$,
\begin{equation*}
    f(\alpha)\coloneqq \dimH\{x\in K:\diml\mus(x)=\alpha\}
\end{equation*}
is the concave conjugate of a differentiable function.
Combining this with \cref{c:mfcor} and involutivity of concave conjugation, we obtain the following result:
\begin{corollary}
    Let $\{S_i\}_{i=1}^3$ be the IFS defined in \cref{e:exifs}.
    Then if $p_2>p_3$, the $L^q$-spectrum $\tau(\mus,q)$ is differentiable for any $q\in\R$.
\end{corollary}
This answers some of the questions raised in \cite{dn2017}.

Finally, in \cref{s:fn-ex}, we investigate some specific families of IFSs to illustrate these results; notably, we give an in-depth analysis of the IFS given in \cref{e:exifs}.
In fact, every example in that section has a finite transition graph: this is equivalent to the generalized finite condition of Lau and Ngai~\cite{ln2007} holding with respect to an open interval (see~\cite[Thm.~3.4]{hhr2021} and \cref{r:fnc-d} for a proof).
Moreover, when $K$ is a convex set, a recent result gives that the weak separation condition is equivalent to the finiteness of the transition graph~\cite[Thm.~4.4]{hhr2021} (see also \cite{fen2016}).
In general, the author believes this to be true without any convexity assumption on $K$:
\begin{conjecture}
    Let $\{S_i\}_{i\in\mathcal{I}}$ be an IFS in $\R$ with transition graph $\mathcal{G}$.
    Then $\{S_i\}_{i\in\mathcal{I}}$ satisfies the weak separation condition if and only if $\mathcal{G}$ is finite.
\end{conjecture}
The results obtained in this paper under the weak separation condition, and the similar strength to results proven under various finite type conditions, provide some more evidence towards this equivalence in general.
\subsection{Limitations and future work}
We note here that the \cref{c:mfcor} is not a dichotomy.
While the non-existence of cycles outside the transition graph guarantees that the multifractal formalism holds, the converse need not hold.
We have examples of measures satisfying the open set condition (with respect to an open set that is not an open interval) with cycles outside the essential class, while the open set condition guarantees that the multifractal formalism does hold.
This situation is likely a by-product of the net interval construction, since our perspective is always with respect to images of the entire interval $[0,1]$.
However, there are also cases such as the Bernoulli measure associated with the IFS $\{x\mapsto \rho x,x\mapsto \rho x+(1-\rho)\}$ where $1/\rho$ is the Golden mean.
In this situation, the attractor is the entire interval $[0,1]$ so that the net interval construction is a natural choice.
Here, even though the $L^q$-spectrum contains a point of non-differentiability at some $q_0<0$ and contains a cycle not contained in the essential class, the measure still satisfies the multifractal formalism~\cite{fen2005}.
These phenomena, and other related special cases, are studied in recent work of Hare, Hare, and Shen~\cite{hhs2021}.

More work is needed to address the general case.
In \cite{rut2021a}, the author investigates the multifractal analysis of measures when the transition graph is finite to provide a more detailed understanding of such examples.
In particular, we obtain a greater understanding of the multifractal formalism outside the essential class as a continuation of our analysis here.

\subsection{Notational conventions}
We briefly mention here some of the conventions we use through out the document.
Given any set $X$, we write $\# X$ to denote the cardinality of $X$.
The set $\R$ is always the metric space equipped with the usual Euclidean metric.
The set $\N$ is the set of natural numbers beginning at $1$.
The set $B(x,t)$ is always a closed ball about $x$ with radius $t$, and $U(x,t)$ denotes the open ball.
Let $E,F\subseteq\R$ be Borel sets.
We denote by $\diam(E)=\sup\{|x-y|:x,y\in E\}$ and $\dist(E,F)=\inf\{|x-y|:x\in E,y\in F\}$.
Given $\delta>0$, we write $E^{(\delta)}=\{x\in\R:\dist(x,E)\leq\delta\}$.
By $E^\circ$, we mean the topological interior of $E$.

Boldface quantities are typically vectors.
If $M$ is a square matrix, we denote by $\spr(M)$ the spectral radius of $M$ and $\norm{M}=\sum_{i,j}|M_{i,j}|$ the matrix 1-norm.
If $\bm{v}$, $\bm{w}$ are vectors with the same dimension, we write $\bm{v}\preccurlyeq\bm{w}$ if $\bm{v}_i\leq\bm{w}_i$ for each $i$.
All matrices in this document are non-negative.

Given families of real numbers $(a_i)_{i\in I}$ and $(b_i)_{i\in I}$, we write $a_i\asymp b_i$ if there exist constants $c_1,c_2>0$ such that $c_1a_i\leq b_i\leq c_2a_i$ for all $i\in I$.

The maps $\{S_i\}_{i\in\mathcal{I}}$ always denotes an iterated function system.
We assume that $\#\mathcal{I}\geq 2$ and its attractor $K$ is not a singleton.
Sets denoted by $\Delta$ are closed intervals and often net intervals.
Indices $s,t$ are used to refer to generations and radii of open and closed balls.
Greek letters $\sigma,\tau,\omega,\phi,\xi$ typically refer to words in $\mathcal{I}^*$.
The Greek $\eta$ typically refers to a path in the transition graph.
The character $T$ refers to either a transition matrix or, more occasionally, a similarity map, depending on context.

\subsection{Acknowledgements}
The author would like to thank Kathryn Hare and Kevin Hare for their support and frequent discussions concerning many of the topics in this paper.
The author also thanks an anonymous referee for pointing out an error in a prior version of \cref{ti:multi-wsc}, and for comments which suggested the current version.

\section[Net Intervals]{Iterated function systems through net intervals}\label{s:net-iv}
\subsection{Iterated function systems of similarities in \texorpdfstring{$\R$}{R}}

Let $\mathcal{I}$ be a non-empty finite index set.
By an iterated function system of similarities (IFS) $\{S_{i}\}_{i\in\mathcal{I}}$ we mean a finite set of similarities 
\begin{equation}\label{e:ifs}
    S_{i}(x)=r_{i}x+d_{i}:\mathbb{R}\rightarrow \mathbb{R}\text{ for each } i\in\mathcal{I}
\end{equation}
with $0<\left\vert r_{i}\right\vert <1$.
We say that the IFS is \defn{(positive) equicontractive} if each $r_i=r>0$.

Each IFS generates a unique non-empty compact set $K$ satisfying 
\begin{equation*}
    K=\bigcup_{i\in\mathcal{I}}S_{i}(K).
\end{equation*}
This set $K$ is known as the associated \defn{self-similar set}.
Throughout, we will assume $K$ is not a singleton.
By rescaling and translating the $d_{i}$ if necessary, without loss of generality we may assume the convex hull of $K$ is $[0,1]$.

Given a probability vector $\bm{p}=(p_i)_{i\in\mathcal{I}}$ where $p_i>0$ and $\sum_{i\in\mathcal{I}}p_i=1$, there exists a unique Borel measure $\mus$ with $\supp\mus=K$ satisfying
\begin{equation}\label{e:minv}
    \mus(E) = \sum_{i\in\mathcal{I}}p_i\mus(S_i^{-1}(E))
\end{equation}
for any Borel set $E\subseteq K$.
This measure $\mus$ as known as an associated \defn{self-similar measure}.

Let $\mathcal{I}^*$ denote the set of all finite words on $\mathcal{I}$.
Given $\sigma =(\sigma_{1},\ldots ,\sigma_{j})\in \mathcal{I}^*$, we denote
\begin{equation*}
    \sigma^{-}=(\sigma_{1},\ldots ,\sigma_{j-1})\text{, }S_{\sigma }=S_{\sigma_{1}}\circ \cdots \circ S_{\sigma_{j}}\text{ and }r_{\sigma }=r_{\sigma_{1}}\cdots r_{\sigma_j}.
\end{equation*}
Given $t >0,$ put
\begin{equation*}
    \Lambda_{t}=\{\sigma \in \mathcal{I}^{\ast }:|r_{\sigma }|<t \leq |r_{\sigma^{-}}|\}.
\end{equation*}
We refer to the set of $\sigma \in \Lambda_{t}$ as the \defn{words of generation $t$}.
We remark that in the literature it is more common to see this defined by the rule $|r_{\sigma }|\leq t <|r_{\sigma^{-}}|$.
The two choices are essentially equivalent, but this choice is more convenient for our purposes.

\subsection{Neighbour sets}
The notions of net intervals and neighbour sets were introduced in \cite{fen2003} and \cite{hhs2018}.
In~\cite{hhr2021}, these notions were extended to an arbitrary IFS, and we present those definitions here.
We then continue the discussion to define the children of a net interval, and show in \cref{t:ttype} that the children depend only on the neighbour set of the parent.

Let $h_{1},\ldots ,h_{s(t)}$ be the collection of distinct elements of the set $\{S_{\sigma }(0),S_{\sigma }(1):\sigma \in \Lambda_{t}\}$ listed in strictly ascending order; we refer to this set as the \defn{endpoints of generation} $t$.
Set
\begin{equation*}
    \mathcal{F}_{t}=\{[h_{j},h_{j+1}]:1\leq j<s(t)\text{ and }(h_{j},h_{j+1})\cap K\neq \emptyset \}.
\end{equation*}
Elements of $\mathcal{F}_{t}$ are called \defn{net intervals of generation} $t$.
Write $\mathcal{F}=\bigcup_{t>0}\mathcal{F}_t$ to denote the set of all possible net intervals.

Suppose $\Delta \in \mathcal{F}$.
We denote by $T_{\Delta }$ the unique contraction $T_{\Delta }(x)=rx+a$ with $r>0$ such that 
\begin{equation*}
    T_{\Delta }([0,1])=\Delta.
\end{equation*}%
Of course, $r=\diam(\Delta)$ and $a$ is the left endpoint of $\Delta$.

\begin{definition}\label{d:nb}
    We will say that a similarity $f(x)=Rx+a$ is a \defn{neighbour} of $\Delta \in \mathcal{F}_{t}$ if there exists some $\sigma \in \Lambda_{t}$ such that $S_{\sigma }(K)\cap\Delta^\circ\neq\emptyset $ and $f=T_{\Delta }^{-1}\circ S_{\sigma }$.
    In this case, we also say that $S_{\sigma }$ \defn{generates} the neighbour $f$.
    The \defn{neighbour set} of $\Delta $ is the maximal set 
    \begin{equation*}
        \vs_t(\Delta )=\{f_{1},\ldots ,f_{m}\}
    \end{equation*}
    where each $f_{i}=T_{\Delta }^{-1}\circ S_{\sigma_{i}}$ is a distinct neighbour of $\Delta$.
\end{definition}
Since $K=\bigcup_{\sigma\in\Lambda_t}S_\sigma(K)$, every net interval has a non-empty neighbour set.

If $\sigma$ generates a neighbour of $\Delta$, then $S_\sigma([0,1])\supseteq\Delta$.
When the generation of $\Delta$ is implicit, we will simply write $\vs(\Delta)$.
For notational convenience, we define the quantity $\lm(\Delta)=\max\{|R|:\{x\mapsto Rx+a\}\in \vs(\Delta)\}$, which depends only on $\vs(\Delta)$.
\begin{remark}
    For an IFS of the form $\{S_i(x)=r x+d_i\}_{i\in\mathcal{I}}$ where $0<r<1$ is fixed, the notion of a neighbour set is related to the characteristic vector of Feng~\cite{fen2003}.
    We describe the equivalence here.

    Let $\Delta=[a,b]\in\mathcal{F}_t$ be some net interval and let $n$ be such that $r^n<t\leq r^{n-1}$.
    Let $\sigma_1,\ldots,\sigma_m$ generate distinct neighbours of $\Delta$, so that $r_{\sigma_i}=r^n$ for each $1\leq i\leq m$.
    Then the (reduced) characteristic vector of $\Delta$ (see~\cite[Sec.~2]{fen2003} for notation) is determined by
    \begin{align*}
        \ell_n(\Delta) &= r^{-n}\diam(\Delta) & V_n(\Delta) &= \{r^{-n}(a-S_{\sigma_i}(0)):1\leq i\leq m\}.
    \end{align*}
    whereas the neighbour set of $\Delta$ is given by
    \begin{align*}
        \vs(\Delta) &= \{T_\Delta^{-1}\circ S_{\sigma_i}\} = \{x\mapsto \frac{S_{\sigma_i}(x)-a}{\diam(\Delta)}\}\\
                    &= \{x\mapsto \frac{x}{r^{-n}\diam(\Delta)}+\frac{S_{\sigma_i}(0)-a}{\diam(\Delta)}\}.
    \end{align*}
    Thus, when the IFS has a common positive contraction ratio, our neighbour set construction can be interpreted directly as a normalized version of Feng's characteristic vector.

    When the IFS has arbitrary contraction ratios, there is no clear choice of normalization factor analogous to $\ell_n(\Delta)$ that is uniform across all net intervals $\Delta\in\mathcal{F}_t$.
    This issue is resolved by normalizing directly by $\diam(\Delta)$, but now it is no longer clear how to define the children of a net interval in a global way.
    Instead, a local definition for the children of net intervals, and the analogue of~\cite[Lem.~2.1]{fen2003}, are given in \cref{ss:ch-netiv}.
\end{remark}

Neighbour sets of net intervals are relevant in the sense that they completely determine the local geometry of $K$ in the net interval, as well as the behaviour of associated self-similar measures on Borel subsets of the net interval.
To be precise, we have the following lemma:
\begin{lemma}\label{l:nbdet}
    Let $\{S_i\}_{i\in\mathcal{I}}$ be an IFS as in \cref{e:ifs} with attractor $K$ and associated self-similar measure $\mus$.
    Suppose $\Delta_1,\Delta_2$ are net intervals with $\vs(\Delta_1)=\vs(\Delta_2)$.
    Then there exists a surjective similarity $g:\Delta_1\cap K\to\Delta_2\cap K$ and constants $c_1,c_2>0$ such that if $E\subseteq\Delta_1$ is any Borel set,
    \begin{equation*}
        c_1\mus(E)\leq \mus(g(E))\leq c_2\mus(E).
    \end{equation*}
\end{lemma}
\begin{proof}
    By definition of the neighbour set, if $\Delta$ is any net interval, we have
    \begin{equation*}
        \Delta\cap K = \bigcup_{f\in\vs(\Delta)}\bigl(T_\Delta\circ f(K)\bigr)\cap\Delta.
    \end{equation*}
    Set $g=T_{\Delta_2}\circ T_{\Delta_1}^{-1}$ so that $g$ is clearly a similarity, and applying this observation to $\Delta_1$ and $\Delta_2$, we have
    \begin{align*}
        g(\Delta_1\cap K) &= \bigcup_{f\in\vs(\Delta_1)}g(T_{\Delta_1}\circ f(K)\cap\Delta_1) = \bigcup_{f\in\vs(\Delta_1)}\bigl(g\circ T_{\Delta_1}\circ f(K)\bigr)\cap g(\Delta_1)\\
                          &= \bigcup_{f\in\vs(\Delta_2)}\bigl(T_{\Delta_2}\circ f(K)\bigr)\cap\Delta_2 = \Delta_2\cap K.
    \end{align*}
    Thus $g$ is a surjective with the correct image.

    We now verify the measure property.
    By the invariant property of the self-similar measure \cref{e:minv}, if $\Delta\in\mathcal{F}_t$ is any net interval and $E\subseteq\Delta$ is any Borel set,
    \begin{align*}
        \mus(E) &= \sum_{\sigma\in\Lambda_{t}}p_\sigma\mus\circ S_{\sigma}^{-1}(E)= \sum_{f\in\vs(\Delta)}\mus\bigl(f^{-1}\circ T_{\Delta}^{-1}(E)\bigr)\sum_{\substack{\sigma\in\Lambda_t\\\sigma\text{ generates }f}}p_\sigma.
    \end{align*}
    Since $f$ is a neighbour of $\Delta$, there is at least one $\sigma$ generating $f$.
    In particular, say $\Delta_1\in\mathcal{F}_{t_1}$ and $\Delta_2\in\mathcal{F}_{t_2}$, write $\vs(\Delta_1)=\vs(\Delta_2)=\{f_1,\ldots,f_m\}$, and set for each $1\leq i\leq m$ and $j=1,2$
    \begin{equation*}
        q_{i,j} \coloneqq  \sum_{\substack{\sigma\in\Lambda_{t_j}\\\sigma\text{ generates }f_i}}p_\sigma>0.
    \end{equation*}
    Set $c_1 = \min\{q_{i,2}/q_{i,1}:1\leq i\leq m\}$.
    We then have for $E\subseteq\Delta_1$ that $g(E)\subseteq\Delta_2$ so that
    \begin{align*}
        \mus(g(E)) &= \sum_{i=1}^m\mus\bigl(f_i^{-1}\circ T_{\Delta_2}^{-1}\circ g(E)\bigr)q_{i,2}\\
                   &\geq c_1\sum_{i=1}^m \mus\bigl(f_i^{-1}\circ T_{\Delta_1}^{-1}(E)\bigr) q_{i,1}= c_1\mus(E).
    \end{align*}
    Similarly, we have $\mus(g(E))\leq c_2\mus(E)$ where $c_2=\min\{q_{i,1}/q_{i,2}:1\leq i\leq m\}$.
\end{proof}
We will revisit these ideas in \cref{ss:tg-im}.

\subsection{Children of net intervals}\label{ss:ch-netiv}
Let $\Delta\in\mathcal{F}$ have neighbour set $\{f_1,\ldots,f_m\}$, and for each $i$, let $S_{\sigma_i}$ generate the neighbour $f_i$ (recall that this means that $S_{\sigma_i}(K)\cap\Delta^\circ\neq\emptyset$ and $f_i=T_{\Delta}^{-1}\circ S_{\sigma_i}$).
\begin{definition}\label{d:tg}
    We define the \defn{ancestral generation} of $\Delta$, denoted $\ag(\Delta)$, and the \defn{transition generation} of $\Delta$, denoted $\tg(\Delta)$, to be positive real values such that
    \begin{equation*}
       \bigcap_{i=1}^m (|r_{\sigma_i}|,|r_{\sigma_i^-}|]=(\tg(\Delta),\ag(\Delta)].
    \end{equation*}
\end{definition}

Note that $0<\tg(\Delta)\leq 1$; if $\Delta=[0,1]$, we say $\ag(\Delta)=\infty$.
It is straightforward to verify that
\begin{itemize}[nl]
    \item $\tg(\Delta)=\lm(\Delta)\cdot\diam(\Delta)$,
    \item $t\in (\tg(\Delta),\ag(\Delta)]$,
    \item for any $s\in (\tg(\Delta),\ag(\Delta)]$, $\Delta\in\mathcal{F}_{s}$ and $\vs_{s}(\Delta)=\vs_t(\Delta)$, and
    \item if $s\notin (\tg(\Delta),\ag(\Delta)]$, either $\Delta\notin\mathcal{F}_{s}$ or $\vs_{s}(\Delta)\neq \vs_t(\Delta)$.
\end{itemize}

Let $t>0$ and $\Delta\in\mathcal{F}_t$.
Let $(\Delta_1,\ldots,\Delta_n)\in\mathcal{F}_{\tg(\Delta)}$ be the distinct net intervals, ordered from left to right, of generation $\tg(\Delta)$ contained in $\Delta$.
Note that either $n>1$ or if $n=1$, then $\vs(\Delta)\neq \vs(\Delta_1)$.
Then we call the tuple $(\Delta_1,\ldots,\Delta_n)$ the \defn{children} of $\Delta\in\mathcal{F}_t$.
Note that for any child $\Delta_i$ of $\Delta$, $\ag(\Delta_i)=\tg(\Delta)$.

Similarly, we define the \defn{parent} of $\Delta\in\mathcal{F}_t$ to be the net interval $\widehat\Delta\in\mathcal{F}_{s}$ with $s > t$ where $\Delta$ is a child of $\widehat\Delta$.
\begin{remark}
    One way to think about the children of a net interval is as follows.
    Enumerate the points $\left\{\prod_{i\in\mathcal{I}}|r_i^{a_i}|:a_i\in\{0\}\cup\N\right\}$ in decreasing order $(t_i)_{i=1}^\infty$.
    Since $\tg(\Delta)=|r_\sigma|$ for some $\sigma\in\mathcal{I}^*$, the transitions to new generations must happen at some $t_i$.
    However, if $\Delta\in\mathcal{F}_{t_k}$, it may not hold that $\tg(\Delta)=t_{k+1}$.
    The children are the net intervals in generation $t_{m}$ where $m\geq k+1$ is minimal such that either $\Delta\notin\mathcal{F}_{t_m}$ or $\vs_{t_m}(\Delta)\neq \vs_{t_k}(\Delta)$.
    
    If the IFS is of the form $\{x\mapsto rx+d_i\}_{i\in\mathcal{I}}$ for some fixed $0<r<1$ and $\Delta\in\mathcal{F}_{r^n}$, then $\tg(\Delta)=r^{n+1}$.
\end{remark}
\begin{example}
    For a worked example of neighbour set and children computations of a non-commensurable IFS, see \cref{ss:ifs-noncomm}.
\end{example}

A key feature of the preceding definitions is that, in a sense that will be made precise, the neighbour set of some net interval $\Delta\in\mathcal{F}_\alpha$ completely determines the placement and the neighbour set of each child of the net interval.
\begin{definition}\label{d:pos-idx}
    Suppose $\Delta=[a,b]\in\mathcal{F}$ has children $(\Delta_1,\ldots,\Delta_n)$ in generation $\tg(\Delta)$.
    For some fixed child $\Delta_i=[a_i,b_i]$, we define the \defn{position index} $q(\Delta,\Delta_i)=(a_i-a)/\diam(\Delta)$.
\end{definition}
One purpose of the position index is to distinguish the children of $\Delta$ which have the same neighbour set.

We have the following basic result.
The insight behind this result is straightforward.
The children of a net interval are determined precisely by the words which generate the neighbours of maximal length.
Up to normalization by the position of $\Delta$, these correspond uniquely to the neighbours of $\Delta$ with maximal contraction factor.
\begin{theorem}\label{t:ttype}
    Let $\{S_i\}_{i\in\mathcal{I}}$ be an arbitrary IFS.
    Let $\Delta\in\mathcal{F}_t$ have children $(\Delta_1,\ldots,\Delta_{n})$ in $\mathcal{F}_{\tg(\Delta)}$.
    Then for any $\Delta'\in\mathcal{F}_s$ with $\vs(\Delta)=\vs(\Delta')$ and children $(\Delta_1',\ldots,\Delta_{n'}')$ in $\mathcal{F}_{\tg(\Delta')}$, we have that $n=n'$ and for any $1\leq i\leq n$,
    \begin{enumerate}[nl,r]
        \item $\vs(\Delta_i')=\vs(\Delta_i)$,
        \item $q(\Delta',\Delta_i')=q(\Delta,\Delta_i)$,
        \item $\frac{\diam(\Delta_i')}{\diam(\Delta')}=\frac{\diam(\Delta_i)}{\diam(\Delta)}$, and
        \item $\frac{\tg(\Delta_i)}{\tg(\Delta)}=\frac{\tg(\Delta_i')}{\tg(\Delta_i)}$.
    \end{enumerate}
\end{theorem}
\begin{proof}
    Given a map $f(x)=rx+d$, we set $R(f)=|r|$.

    Write $\vs(\Delta')=\vs(\Delta)=\{f_1,\ldots,f_m\}$, and let
    \begin{align*}
        \mathcal{W}' &= \{T_{\Delta'}\circ f_i:R(f_i)=\lm(\Delta'),1\leq i\leq m\}\\
        \mathcal{W} &= \{T_{\Delta}\circ f_i:R(f_i)=\lm(\Delta),1\leq i\leq m\}
    \end{align*}
    denote the corresponding sets of neighbours corresponding to functions with maximal contraction factor, where $\lm(\Delta')=\lm(\Delta)$.
    Then let
    \begin{align*}
        \mathcal{C}' &= \bigl\{S_\tau:\tau\in\Lambda_{\tg(\Delta')},S_\tau(K)\cap (\Delta')^\circ\neq\emptyset\bigr\}\\
        \mathcal{C} &= \bigl\{S_\tau:\tau\in\Lambda_{\tg(\Delta)},S_\tau(K)\cap\Delta^\circ\neq\emptyset\bigr\}.
    \end{align*}
    In other words, $\mathcal{C}$ is the set of words of generation $\tg(\Delta)$ which contribute to some child of $\Delta$, and similarly for $\Delta'$.
    Using the observation that the only new words are those which are one-level descendants of those which generate neighbours of maximal length, we have
    \begin{align}
        \mathcal{C} &= \{f\circ S_j:f\in \mathcal{W},f\circ S_j(K)\cap\Delta^\circ\neq\emptyset\}\cup\{T_{\Delta}^{-1}\circ f_i:R(f_i)\neq \lm(\Delta)\}\nonumber\\
                    &= \{T_{\Delta}\circ T_{\Delta'}^{-1}\circ f:f\in \mathcal{C}'\}.\label{e:c-cprime}
    \end{align}
    Note that, in the above set of equalities, we use the fact that for $f\in\mathcal{W}$
    \begin{align*}
        f\circ S_j(K)\cap\Delta^\circ\neq\emptyset &\iff T_\Delta^{-1}\circ f\circ S_j(K)\cap (0,1)\neq\emptyset\\
                                                   &\iff T_{\Delta'}\circ T_\Delta^{-1}\circ f\circ S_j(K)\cap(\Delta')^\circ\neq\emptyset
    \end{align*}
    where $T_{\Delta'}\circ T_\Delta^{-1}\circ f\in \mathcal{W}'$.

    Write $\Delta=[a,b]$ and $\Delta'=[a',b']$.
    Now consider the set $H=\{a,b\}\cup \{f(0),f(1):f\in \mathcal{C}\}\cap\Delta$ so that $H$ is the set of all endpoints of generation $\tg(\Delta)$ contained in $\Delta$.
    Then if $H'=\{a',b'\}\cup \{f(0),f(1):f\in \mathcal{C}'\}\cap\Delta'$, we observe by \cref{e:c-cprime} that $T_{\Delta'}^{-1}(H')=T_\Delta^{-1}(H)$.
    Let $a=h_1<\cdots<h_{k+1}=b$ denote the ordered elements of $H$ and $a'=h_1'<\cdots<h_{k+1}'=b'$ the ordered elements of $H'$ where $k=|H|-1=|H'|-1$.
    By \cref{l:nbdet}, $(h_i,h_{i+1})\cap K\neq\emptyset$ if and only if $(h_i',h_{i+1}')\cap K\neq\emptyset$.
    Thus the children of $\Delta$ are $\{[h_i,h_{i+1}]:(h_i,h_{i+1})\cap K\neq\emptyset\}$ and the children of $\Delta'$ are $\{T_{\Delta'}\circ T_\Delta^{-1}([h_i,h_{i+1}]):(h_i,h_{i+1})\cap K\neq\emptyset\}$, so $k=n=n'$.

    Now fix some $1\leq i\leq n$.
    Note that $T_\Delta\circ T_{\Delta'}^{-1}(\Delta_i')=\Delta_i$ so that $T_{\Delta_i}^{-1}\circ T_\Delta\circ T_{\Delta'}^{-1}=T_{\Delta_i'}^{-1}$.
    \begin{enumerate}[nl,r]
        \item By direct computation,
            \begin{align*}
                \vs(\Delta_i) &= \{T_{\Delta_i}^{-1}\circ f:f\in\mathcal{C},f(K)\cap\Delta_i^\circ\neq\emptyset\}\\
                              &= \{T_{\Delta_i}^{-1}\circ T_{\Delta}\circ T_{\Delta'}^{-1}\circ f:f\in\mathcal{C}',T_{\Delta}\circ T_{\Delta'}^{-1}\circ f(K)\cap\bigl(T_\Delta\circ T_{\Delta'}^{-1}(\Delta_i')\bigr)^\circ\neq\emptyset\}\\
                              &= \{T_{\Delta_i'}^{-1}\circ f:f\in\mathcal{C}',f(K)\cap(\Delta_i')^\circ\neq\emptyset\}=\vs(\Delta_i')
            \end{align*}
        \item Since the $T_{\Delta}$ are isometries, $q(\Delta,\Delta_i) = \frac{h_i-h_1}{\diam(\Delta)}=T_{\Delta}^{-1}(h_i)$ since $T_{\Delta}^{-1}(h_1)=0$.
            Then the result follows since $T_{\Delta}^{-1}(h_i)=T_{\Delta'}^{-1}(h_i')$.
        \item We have
            \begin{equation*}
                \frac{\diam(\Delta_i)}{\diam(\Delta)}=\diam(T_\Delta^{-1}(\Delta_i))=\diam(T_{\Delta'}^{-1}(\Delta_i'))=\frac{\diam(\Delta_i')}{\diam(\Delta')}
            \end{equation*}
        \item Recall that for an arbitrary net interval, $\tg(\Delta_0) = \lm(\Delta_0)\cdot\diam(\Delta_0)$ where $\lm(\Delta_0)$ depends only on $\vs(\Delta_0)$.
            Apply (i) and (iii).
    \end{enumerate}
    We thus have the desired result.
\end{proof}

\subsection{The transition graph of an iterated function system}\label{ss:tgraph}
In the context of \cref{t:ttype}, to understand the behaviour of the IFS, it is in a sense sufficient to track the behaviour of the neighbour sets.
Thus, we construct the \defn{transition graph} of the IFS.
The transition graph is a directed graph $\mathcal{G}(\{S_i\}_{i\in\mathcal{I}})$, possibly with loops and multiple edges, (denoted by $\mathcal{G}$ when the IFS is clear from the context) defined as follows.
The vertex set of $\mathcal{G}$, denoted $V(\mathcal{G})$, is $\{\vs(\Delta):\Delta\in\mathcal{F}\}$, the set of distinct neighbour sets.
The edge set of $\mathcal{G}$, denoted $E(\mathcal{G})$, is a set of triples $(v_1,v_2,q)$ where $v_1$ is the source vertex, $v_2$ is the target vertex, and $q$ is the edge label to distinguish multiple edges.
The edges are given as follows: for each net interval $\Delta\in\mathcal{F}_t$ with children $(\Delta_1,\ldots,\Delta_m)$ and for each $i$, we introduce an edge $e=(\vs_t(\Delta),\vs_{\tg(\Delta)}(\Delta_i),q(\Delta,\Delta_i))$.
By \cref{t:ttype}, this construction is well-defined since it depends only on the neighbour set of $\Delta$.

An \defn{(admissible) path} $\eta$ in $\mathcal{G}$ is a sequence of edges $\eta=(e_1,\ldots,e_n)$ in $\mathcal{G}$ where the target of $e_i$ is the source of $e_{i+1}$.
A path in $\mathcal{G}$ is a \defn{cycle} if the path begins and ends at the same vertex.

We can encode the behaviour of the IFS symbolically using the transition graph.
Given $\Delta\in\mathcal{F}_t$, consider the sequence $(\Delta_0,\ldots,\Delta_n)$ where $\Delta_0=[0,1]$, $\Delta_n=\Delta$, and each $\Delta_i$ is a child of $\Delta_{i-1}$.
Then the \defn{symbolic representation} of $\Delta$ is the path $\eta=(e_1,\ldots,e_{n})$ of $G$ where for each $1\leq i\leq n$
\begin{equation*}
    e_i=\bigl(\vs(\Delta_{i-1}),\vs(\Delta_i),q(\Delta_{i-1},\Delta_i)\bigr).
\end{equation*}
Conversely, if $\eta$ is any admissible path, we say that $(\Delta_i)_{i=0}^k$ is a \defn{(net interval) realization} of $\eta$ if
\begin{itemize}[nl]
    \item each $\Delta_i$ is a child of $\Delta_{i-1}$, and
    \item each $e_i=(\vs(\Delta_{i-1}),\vs(\Delta_i),q(\Delta_{i-1},\Delta_i))$.
\end{itemize}
By construction, every admissible path has a net interval realization.

Now let $x\in K$ be arbitrary and let $(\Delta_i)_{i=0}^\infty$ be a sequence of nested intervals where $\Delta_0=[0,1]$ and $\Delta_{i+1}$ a child of $\Delta_i$ and $\{x\}=\bigcap_{i=1}^\infty\Delta_i$.
The symbolic representation of $x$ corresponding to sequence $(\Delta_i)_{i=0}^\infty$ is the infinite path $(e_i)_{i=1}^\infty$ where for each $n$, $(e_1,\ldots,e_n)$ is the symbolic representation of $\Delta_n$.
The symbolic representation uniquely determines $x$, but if $x$ is an endpoint of some net interval, it can happen that there are two distinct symbolic representations.

Suppose $\{S_i\}_{i\in\mathcal{I}}$ is of the form $\{x\mapsto r x+d_i\}_{i\in\mathcal{I}}$ where $0<r<1$.
Then if $\Delta\in\mathcal{F}_t$ is any net interval with symbolic representation $\eta=(e_1,\ldots,e_n)$, $\tg(\Delta)=r^n$ and $r^n< t\leq r^{n-1}$.
In other words, given the symbolic representation, we can approximate the generation of $\Delta$.

However, when the IFS is not of this form, paths with the same length can result in net intervals in substantially different generations, and if the contraction ratios are not logarithmically commensurable (i.e. $\log r_i/\log r_j\in\Q$ for any $i,j\in\mathcal{I}$), there is no way to resolve this in a uniform way.
Thus in order to approximate the change in generation along a path in the transition graph, it is necessary to assign distinct values to the edges in the transition graph.

\begin{definition}\label{d:e-len}
    Let $\mathcal{G}$ be the transition graph of an IFS.
    We define the \defn{edge length function} $L:E(\mathcal{G})\to(0,1)$ as follows.
    For a particular edge $e$, let the source and target be given by $v_1$ and $v_2$, where $v_i=\vs(\Delta_i)$ for some $\Delta_1$ the parent of $\Delta_2$, and define $L(e)=\tg(\Delta_2)/\tg(\Delta_1)$.
\end{definition}
This function is well-defined by \cref{t:ttype}.
When $\eta=(e_1,\ldots,e_n)$ is an admissible path, we say $L(\eta)=L(e_1)\cdots L(e_n)$.
\begin{remark}
    If $\{S_i\}_{i\in\mathcal{I}}$ is of the form $\{x\mapsto r x+d_i\}_{i\in\mathcal{I}}$ where $0<r<1$, then $L(e)=r$ for any edge $e\in E(\mathcal{G})$.
\end{remark}
The main point here is that if $\Delta\in\mathcal{F}_t$ is any net interval with symbolic representation $\eta$, then $L(\eta)\asymp t$ with constants of comparability not depending on $\Delta$.
While the above choice of the length for an edge is not unique with this property, a straightforward argument shows that any such function must agree with $L$ on any cycle.

\subsection{Encoding the invariant measure by the transition graph}\label{ss:tg-im}
Given an IFS $\{S_i\}_{i\in\mathcal{I}}$ with a corresponding invariant measure $\mus$, we are interested in formulas for computing or approximating $\mus(E)$ where $E\subseteq K$ is an arbitrary Borel set.
When $\{S_i\}_{i\in\mathcal{I}}$ satisfies the strong separation condition (that is, for $i\neq j$, $S_i(K)$ and $S_j(K)$ are disjoint), this is straightforward since $\mus(S_\sigma(K))=p_\sigma$.
However, when images of $K$ overlap, such a formula no longer holds.

The net interval construction can be thought of as a way of converting the behaviour of the IFS on overlapping images of $K$ into behaviour on net intervals, which are disjoint except on a countable set (which has $\mus$-measure 0).
It turns out that one may also encode the dynamics of the invariant measure $\mus$ using products of matrices.
This technique was developed in the equicontractive case for IFS of the form $\{x\mapsto rx+d_i\}_{i\in\mathcal{I}}$ with $0<r<1$ by Feng~\cite{fen2003}, and extended to IFS which satisfy the finite type condition~\cite{hhs2018}.
Using similar techniques, we describe here how to generalize this construction to an arbitrary IFS.

Let $\{S_i\}_{i\in\mathcal{I}}$ be an IFS and $\mus$ the self similar measure associated to probabilities $\{p_i\}_{i\in\mathcal{I}}$.
The main mechanism to compute the approximate measure of net intervals is through transition matrices.
Recall that $\mathcal{G}$ has vertex set $V(\mathcal{G})=\{\vs(\Delta):\Delta\in\mathcal{F}\}$.
Fix some total ordering on the set of all neighbours $\{f:f\in\vs(\Delta),\Delta\in\mathcal{F}\}$.

Let $e\in E(\mathcal{G})$ be a fixed edge with source $v_1$ and target $v_2$.
Suppose $\Delta_1\supseteq\Delta_2$ are net intervals such that $\Delta_1$ is the parent of $\Delta_2$ and $e=(\vs(\Delta_1),\vs(\Delta_2),q(\Delta_1,\Delta_2))$.
Suppose the neighbour sets are given by $\vs(\Delta_1)=\{f_1,\ldots,f_m\}$ and $\vs(\Delta_2)=\{g_1,\ldots,g_{n}\}$ where $f_1<\cdots<f_m$ and $g_1<\cdots<g_{n}$.
We then define the \defn{transition matrix} $T(e)$ as the non-negative $m\times n$ matrix given by
\begin{equation}\label{e:trmat}
    T(e)_{i,j}=\frac{\mus(g_j^{-1}((0,1))}{\mus(f_i^{-1}((0,1))}\cdot p_\ell
\end{equation}
if there exists an index $\ell\in\mathcal{I}$ such that $f_i$ is generated by $\sigma$ and $g_j$ is generated by $\sigma \ell$; otherwise, set $T(e)_{i,j}=0$.
This is well-defined since a neighbour $f$ has $f^{-1}((0,1))\cap K\neq\emptyset$ by definition.
Recall that if $\sigma'$ generates any neighbour of $\Delta_2$, then necessarily $\sigma'=\sigma \ell$ for some $\sigma$ which generates a neighbour of $\Delta_1$; thus, every column of $T(e)$ has a positive entry.
However, it may not hold that each row of $T(e)$ has a positive entry.

It is clear from \cref{t:ttype} that this definition depends only on the edge $e$.
If $\eta=(e_1,\ldots,e_n)$ is an admissible path, we define $T(\eta)=T(e_1)\cdots T(e_n)$.
\begin{example}
    See \cref{ss:ifs-noncomm} and \cref{f:tgraph-m} for a complete transition graph example.
\end{example}

Throughout, we will denote by $\norm{T}=\sum_{i,j}|T_{ij}|$ to denote the matrix $1$-norm.
Suppose $\Delta\in\mathcal{F}_t$ is an arbitrary net interval.
From the defining identity of the self-similar measure,
\begin{equation*}
    \mus(\Delta)=\sum_{\sigma\in\Lambda_t}p_\sigma\mus(S_\sigma^{-1}(\Delta))
\end{equation*}
where, since $\mus$ is non-atomic, the summation may be taken over $\sigma$ such that $S_\sigma^{-1}(\Delta^\circ)\cap K$ is non-empty.
Note that $S_\sigma^{-1}(\Delta^\circ)=S_\sigma^{-1}\circ T_\Delta((0,1))=f^{-1}((0,1))$ where $f\in \vs(\Delta)$.
We thus have
\begin{equation}\label{e:mu-f}
    \mus(\Delta) = \sum_{f\in \vs(\Delta)}\mus(f^{-1}((0,1)))\sum_{\substack{\sigma\in\Lambda_t\\\sigma\text{ generates }f}}p_\sigma.
\end{equation}
Let $\vs(\Delta)=\{f_1,\ldots,f_m\}$ with $f_1<\cdots<f_m$; then, we denote the \defn{vector form} of $\mus$ by $\muv(\Delta)=(q_1,\ldots,q_m)$ where
\begin{equation*}
    q_i = \mus(f_i^{-1}((0,1)))\sum_{\substack{\sigma\in\Lambda_t\\\sigma\text{ generates }f_i}}p_\sigma.
\end{equation*}
In particular, $\muv(\Delta)$ is a strictly positive vector for any $\Delta$, and $\mus(\Delta)=\norm{\muv(\Delta)}$.

With this notation, we have the following theorem:
\begin{theorem}\label{t:approx}
    Let $\{S_i\}_{i\in\mathcal{I}}$ have associated self-similar measure $\mus$.
    If $\eta$ is any admissible path realized by $(\Delta_i)_{i=0}^m$,
    \begin{equation*}
        \muv(\Delta_0)T(\eta)=\muv(\Delta_m).
    \end{equation*}
\end{theorem}
\begin{proof}
    Suppose $\Delta_0\in\mathcal{F}_t$ and $\Delta_m\in\mathcal{F}_s$.
    Say $\vs(\Delta_0)=\{f_1,\ldots,f_\ell\}$ with $f_1<\cdots<f_\ell$ and $\vs(\Delta_m)=\{g_1,\ldots,g_m\}$ with $g_1<\cdots<g_m$.
    For each $i$, assume $\tau_i$ generates the neighbour $f_i$, and set $\mathcal{A}_{ij}=\{\omega:\tau_i\omega\in\Lambda_s,\tau_i\omega\text{ generates }g_j\}$.
    Then for any $1\leq j\leq m$, we have
    \begin{align*}
        \bigl(\muv(\Delta_0)T(\eta)\bigr)_j &= \sum_{i=1}^\ell \mus(f_i^{-1}((0,1)))\Bigl(\sum_{\substack{\sigma\in\Lambda_t\\\sigma\text{ generates }f_i}}p_\sigma\Bigr)\cdot\Bigl(\sum_{\omega\in\mathcal{A}_{ij}}\frac{\mus(g_j^{-1}((0,1))}{\mus(f_i^{-1}((0,1))}p_\omega\Bigr)\\
                                         &= \mus(g_j^{-1}((0,1))) \sum_{i=1}^\ell\Bigl(\sum_{\substack{\sigma\in\Lambda_t\\\sigma\text{ generates }f_i}}p_\sigma\Bigr)\cdot\Bigl(\sum_{\omega\in\mathcal{A}_{ij}}p_\omega\Bigr)\\
                                         &= \mus(g_j^{-1}((0,1)))\sum_{\substack{\omega\in\Lambda_s\\\omega\text{ generates }g_j}}p_\omega
    \end{align*}
    so that $\muv(\Delta_0)T(\eta)=\muv(\Delta_m)$.
\end{proof}

\section{Iterated function systems satisfying the weak separation condition}\label{s:ifswsc}
We now focus our attention on self-similar measures associated with IFSs satisfying the weak separation condition.
We give a definition which is slightly different than the original~\cite{ln1999}, but is known to be equivalent when $K$ is not a singleton~\cite{zer1996}.
Given a Borel set $E\subset K$ and $t>0$, we define
\begin{align*}
    \Lambda_t(E) &= \{\sigma\in\Lambda_t:S_\sigma(K)\cap E\neq\emptyset\}\\
    \mathcal{S}_t(E) &= \{S_\sigma:\sigma\in\Lambda_t(E)\}
\end{align*}
Let $U(x,t)$ denote the open ball about $x$ with radius $t$.
\begin{definition}
    We say that the IFS $\{S_i\}_{i\in\mathcal{I}}$ satisfies the \defn{weak separation condition} if
    \begin{equation}\label{e:wsc-max}
        \sup_{x\in \R,t>0}\#\mathcal{S}_t(U(x,t))<\infty.
    \end{equation}
\end{definition}
We can obtain an equivalent formulation of the weak separation condition in terms of a variant of the neighbour set which we call the \defn{covering neighbour set}.
Given a net interval $\Delta\in\mathcal{F}_t$, we write $\vsc(\Delta)=\{T_\Delta^{-1}\circ S_\sigma:\sigma\in\Lambda_t,S_\sigma([0,1])\supseteq\Delta\}$.
We refer to elements of $\vsc(\Delta)$ as \defn{covering neighbours}.
Notably, we omit the requirement that a neighbour $f\in\vsc(\Delta)$ has $f(K)\cap(0,1)\neq\emptyset$.
\begin{remark}\label{r:cov-nb}
    We always have $\vs(\Delta)\subseteq\vsc(\Delta)$ with strict inequality possible.
    Moreover, we note that if $\Delta$ and $\Delta'$ are any net intervals with $\vsc(\Delta)=\vsc(\Delta')$, then necessarily $\vs(\Delta)=\vs(\Delta')$ following similar arguments to \cref{l:nbdet} and \cref{t:ttype}.
    Note that the covering neighbour set is taken as the definition of neighbour set in \cite{hhr2021}.
\end{remark}

We have the following characterization, which is \cite[Prop.~4.3]{hhr2021}:
\begin{proposition}[\cite{hhr2021}]\label{p:wsci}
    The IFS $\{S_i\}_{i\in\mathcal{I}}$ satisfies the weak separation condition if and only if
    \begin{equation*}
        \sup_{\Delta\in\mathcal{F}}\#\vsc(\Delta)<\infty.
    \end{equation*}
\end{proposition}
Net intervals for which $\#\vsc(\Delta)$ attain the supremum in \cref{p:wsci} will play an important role in our analysis in this section.

\subsection{The essential class of the transition graph}
Let $\{S_i\}_{i\in\mathcal{I}}$ be an IFS with associated transition graph $\mathcal{G}$.
Recall that in a directed graph $\mathcal{G}$, an induced subgraph $\mathcal{G}'$ is a subgraph for which there exists some set of vertices $H\subseteq V(\mathcal{G})$ such that $\mathcal{G}'$ has vertex set $H$ and edge set composed of every outgoing edge from a vertex in $H$ which connects to another vertex in $H$.
\begin{definition}
    An \defn{essential class} of $\mathcal{G}$ is an induced subgraph $\mathcal{G}'$ of $\mathcal{G}$ such that
    \begin{enumerate}[nl,r]
        \item for any $v,v'\in\mathcal{G}'$, there exists a path from $v$ to $v'$, and
        \item if $v\in\mathcal{G}$ and $v'\in\mathcal{G}'$ and there is a path from $v'$ to $v$, then $v\in\mathcal{G}'$.
    \end{enumerate}
\end{definition}
In a finite graph, there is always at least one essential class~\cite[Lem.~1.1]{sen1981}.
In an infinite graph, there need not be an essential class; moreover, the essential class, if it exists, need not be finite.
When $\mathcal{G}$ has exactly one essential class, we denote it by $\mathcal{G}_{\ess}$.

We have the following basic observation.
The proof of this result is similar to the idea in \cite[Lemma~4.2]{hhr2021}, but we reiterate the aspects of the proof that we need here for clarity.
\begin{proposition}\label{p:wsc-es}
    Let $\{S_i\}_{i\in\mathcal{I}}$ be an IFS satisfying the weak separation condition.
    Then its transition graph $\mathcal{G}$ has a unique essential class.
\end{proposition}
\begin{proof}
    It suffices to show that there exists some vertex $v$ such that if $w$ is any other vertex, there exists an admissible path from $w$ to $v$.
    Then the essential class is the set of all vertices $v'$ for which there is a path from $v$ to $v'$.
    By \cref{p:wsci}, there exists some $t>0$ and net interval $\Delta_0\in\mathcal{F}_t$ such that $\#\vsc(\Delta_0)$ is maximal; let $v\coloneqq \vs(\Delta_0)$.

    Now, let $w\in V(\mathcal{G})$ be arbitrary and $\Delta\in\mathcal{F}$ such that $\vs(\Delta)=w$.
    Since $\Delta^\circ\cap K\neq\emptyset$, there exists some $\sigma\in\mathcal{I}^*$ such that $S_\sigma(K)\subseteq\Delta$ and $r_\sigma>0$.
    Set $\gamma=r_\sigma\cdot t$ and let $\Delta_1\coloneqq S_{\sigma}(\Delta_0)$

    Let $\Delta_0=[a,b]$ have covering neighbours generated by words $\{\omega_1,\ldots,\omega_m\}$ with $\omega_i\in\Lambda_t$.
    By definition of $\gamma $, $\{\sigma \omega_{1},\ldots ,\sigma \omega_{m}\}$ are words of generation $\Lambda_{\gamma }$.
    Note that $(\Delta_1)^\circ\cap K\neq\emptyset$ and that the endpoints of $\Delta_1$ are of the form $S_{\sigma\zeta}(z)$ where $z\in\{0,1\}$ and $\zeta\in\Lambda_t$, so that $\sigma\zeta\in\Lambda_\gamma$.
    In particular, if $\Delta_1\notin\mathcal{F}_\gamma$, then there exists some $\tau\in\Lambda_\gamma$ such that $S_\tau\notin\{S_{\sigma\omega_1},\ldots,S_{\sigma\omega_m}\}$ and $S_\tau([0,1])\supseteq \Delta_1$.
    But then there exists some $\Delta_2\in\mathcal{F}_\gamma$ with $\Delta_2\subseteq\Delta_1\cap S_\tau([0,1])$, where $\Delta_2$ has distinct covering neighbours generated by $\{\omega_1,\ldots,\omega_m\}\cup\{\tau\}$, contradicting the maximality of $\#\vsc(\Delta_0)$.

    Thus $\Delta_1$ is in fact a net interval of generation $\gamma$.
    Moreover, since $r_{\sigma}>0$, we have $T_{\Delta_{1}}=S_{\sigma}\circ T_{\Delta_{0}}$, so that 
    \begin{equation*}
        \vsc(\Delta_{1})=\{T_{\Delta_{1}}^{-1}\circ S_{\sigma \omega_{i}}\}_{i=1}^{m}=\{T_{\Delta_{0}}^{-1}\circ S_{\sigma }^{-1}\circ S_{\sigma }\circ S_{\omega_{i}}\}_{i=1}^{m}=\vsc(\Delta_{0}).
    \end{equation*}
    Thus by \cref{r:cov-nb}, we have $\vs(\Delta_1)=v$ and $\Delta_1\subseteq\Delta$, so that there exists a path from $\vs(\Delta)$ to $\vs(\Delta_1)$, as claimed.
\end{proof}
\begin{definition}
    We say that a point $x\in K$ is an \defn{essential point} if for some symbolic representation $(e_j)_{j=1}^\infty$ of $x$, there exists some $N\in\N$ so that for all $k\geq N$, $e_k\in E(\mathcal{G}_{\ess})$.
    We say that a point $x\in K$ is an \defn{interior essential point} if every symbolic representation has this property.
    We denote the set of all interior essential points by $K_{\ess}$.
    We say a net interval $\Delta\in\mathcal{F}$ is an \defn{essential net interval} if $\vs(\Delta)\in V(\mathcal{G}_{\ess})$.
\end{definition}
If $\Delta$ is an essential net interval, then $\Delta^\circ\cap K\subseteq K_{\ess}$.
Of course, a given path $(e_j)_{j=1}^\infty$ is eventually in the essential class if and only if a single edge is in the essential class.
One may verify that the set of interior essential points is the topological interior of the set of essential points; in particular, the essential points form an open set in $K$.
Interior essential points play an important role in the multifractal analysis of self-similar measures under the weak separation condition.

In the next proposition, we observe that interior essential points are abundant.
\begin{proposition}\label{p:mx-wsc}
    Let $\{S_i\}_{i\in\mathcal{I}}$ be an IFS satisfying the weak separation condition.
    Let $U(x_0,t_0)$ be any open ball which attains the maximal value in \cref{e:wsc-max}.
    Then the following hold:
    \begin{enumerate}[nl,r]
        \item If $\sigma\in\mathcal{I}^*$ is arbitrary, then $S_\sigma(U(x_0,t_0))$ also attains the maximal value in \cref{e:wsc-max}.
        \item $U(x_0,t_0)\cap K$ is contained in a finite union of essential net intervals.
            In particular, $U(x_0,t_0)\cap K\subseteq K_{\ess}$.
    \end{enumerate}
\end{proposition}
\begin{proof}
    To see that $S_\sigma(U(x_0,t_0))$ also attains the maximal value in \cref{e:wsc-max}, if
    \begin{equation*}
        \mathcal{S}_{t_0}(U(x_0,t_0))=\{S_{\phi_1},\ldots,S_{\phi_m}\},
    \end{equation*}
    then $S_{\sigma\phi_i}\in\mathcal{S}_{|r_\sigma|t_0}(S_\sigma(U(x_0,t_0)))$ for each $i$ and $\#\mathcal{S}_{|r_\sigma|t_0}(S_\sigma(U(x_0,t_0)))\geq m$.
    Then equality holds by maximality of $m$.

    We now see (ii).
    By definition of net intervals, we know that for any $t>0$, $U(x_0,t_0)\cap K$ is contained in a finite union of net intervals of generation $t$.
    In particular, it suffices to show that there is some $t_1>0$ such the set
    \begin{equation*}
        \{\Delta\in\mathcal{F}_{t_1}:\Delta\cap U(x_0,t_0)\neq\emptyset\}
    \end{equation*}
    is composed only of essential net intervals.
    Let $\Delta_0$ be a fixed essential net interval and let $\sigma_0\in\mathcal{I}^*$ have $r_{\sigma_0}>0$ and $S_{\sigma_0}([0,1])\subseteq\Delta_0$.
    As argued above, $S_{\sigma_0}(U(x_0,t_0))$ also attains the maximal value in \cref{e:wsc-max}.
    Let
    \begin{equation*}
        H=\{S_\sigma:\sigma\in\Lambda_{r_{\sigma_0}t_0},S_\sigma(K)\cap S_{\sigma_0}(U(x_0,t_0))=\emptyset\}.
    \end{equation*}
    Since $S_{\sigma_0}(U(x_0,t_0))$ is open, there exists some $\epsilon_0>0$ such that for any $\epsilon$ with $|\epsilon|<\epsilon_0$, $S_{\sigma_0}(U(x_0+\epsilon,t_0))$ also attains the maximal value in \cref{e:wsc-max}.
    In particular, if $S_\sigma\in H$ is arbitrary, we in fact have $S_\sigma(K)\cap S_{\sigma_0}(B(x_0,t_0))=\emptyset$.
    Since $H$ is a finite set, take
    \begin{equation*}
        t_1 = \min\bigl\{\min\{\dist(f(K),S_{\sigma_0}(B(x_0,t_0))) : f\in H\},t_0\bigr\}>0.
    \end{equation*}
    It remains to show that such a $t_1$ works.

    Write $\mathcal{S}_{t_0}(U(x_0,t_0))=\{S_{\phi_1},\ldots,S_{\phi_m}\}$ and set
    \begin{equation*}
        F = \{\Delta\in\mathcal{F}_{t_1}:\Delta\cap U(x_0,t_0)\neq\emptyset\}.
    \end{equation*}
    Suppose for contradiction there is some $\Delta\in F$ that is not an essential net interval, and let $\Delta$ have neighbours generated by distinct functions $\{S_{\omega_1},\ldots,S_{\omega_k}\}$ with $\omega_i\in\Lambda_{t_1}$.
    As argued in \cref{p:wsc-es}, since $\Delta_1\coloneqq S_{\sigma_0}(\Delta)$ is not a net interval with neighbour set $\vs(\Delta)$ (or $\Delta_1$ would be a descendant of $\Delta_0$, and hence essential), there exists some $\tau\in \Lambda_{r_{\sigma_0}t_1}$ such that $S_\tau(K)\cap\Delta_1^\circ\neq\emptyset$ and $S_\tau\neq S_{\sigma_0\omega_i}$ for each $1\leq i\leq k$.
    We also observe that
    \begin{equation}\label{e:sigma-ch}
        \{S_{\sigma_0\omega_1},\ldots,S_{\sigma_0\omega_k}\}=\{S_{\sigma_0\xi}:\xi\in\Lambda_{t_1},S_{\sigma_0\xi}(K)\cap\Delta_1^\circ\neq\emptyset\}.
    \end{equation}

    Since $t_1\leq t_0$, let $\tau_1\preccurlyeq\tau$ be the unique prefix in $\Lambda_{r_{\sigma_0} t_0}$.
    Suppose for contradiction $S_{\tau_{1}}(K)\cap S_{\sigma_0}(U(x_0,t_0))\neq\emptyset$.
    Since $S_{\sigma_0}(U(x_0,t_0))$ attains the maximal value in \cref{e:wsc-max}, we have $S_{\tau_{1}}=S_{\sigma_0}\circ S_\omega$ for some $S_\omega\in \mathcal{S}_{r_{\sigma_0} t_0}(S_\sigma(U(x_0,t_0)))$.
    Thus there exists some word $\xi$ such that $S_\tau=S_{\sigma_0}\circ S_\xi$, which contradicts \cref{e:sigma-ch}.
    We thus have that $S_{\tau_{1}}(K)\cap S_{\sigma_0}(U(x_0,t_0))=\emptyset$ so that $S_{\tau_1}\in H$.

    But by definition of $\Delta_1$, we have that $\Delta_1\cap S_{\sigma_0}(U(x_0,t_0))\neq\emptyset$ and $\Delta_1^\circ\cap S_{\tau_1}(K)\neq\emptyset$, so
    \begin{equation*}
        \dist(S_{\tau_1}(K),S_{\sigma_0}(U(x_0,t_0)))<\diam(\Delta_1)\leq t_1,
    \end{equation*}
    contradicting the choice of $t_1$.
    Thus every $\Delta\in F$ is in fact essential, as claimed.
\end{proof}
\begin{remark}\label{r:im-essential}
    In fact, the same proof shows that if $U(x_0,t_0)$ attains the maximal value in \cref{e:wsc-max}, $\Delta\subset U(x_0,t_0)$ is any net interval, and $r_\sigma>0$, then $S_\sigma(\Delta)$ is a net interval with $\vs(\Delta)=\vs(S_\sigma(\Delta))$.
    In particular, $\Delta$ must be an essential net interval.
\end{remark}
\begin{remark}
    In \cref{sss:valid-open}, we show that the converse of (ii) need not hold: there exists some IFS $\{S_i\}_{i\in\mathcal{I}}$ satisfying the weak separation condition and an essential net interval $\Delta$ such that $\Delta\cap K$ is not contained a finite union of balls $U(x_0,t_0)$.
    In the same example, we show that if $W$ is the union of all balls $U(x_0,t_0)$ which attain the maximal value in \cref{e:wsc-max}, then $W\cap K\subsetneq K_{\ess}$.
\end{remark}
\subsection{An important measure approximation lemma}
The following technical lemma is a key approximation property for measures satisfying the weak separation condition, and the main factor behind the regularity of the measure on the essential class.
Note the similarity of the result to the weak separation ``counting'' results; see, for example, Feng and Lau~\cite[Prop.~4.1]{fl2009}.
\begin{lemma}\label{l:bmeas}
    Suppose the IFS $\{S_i\}_{i\in\mathcal{I}}$ satisfies the weak separation condition, and let $v\in V(\mathcal{G}_{\ess})$ be fixed.
    Then there exist constants $c,C>0$ (depending on $v$) such that for any ball $B(x,t)$ with $\mus(B(x,t))>0$, there exists $t\geq s\geq c t$ and $\Delta\in\mathcal{F}_s$ such that $\Delta\subseteq B(x,2 t)$, $\vs(\Delta)=v$, and $\muv(\Delta)_j\geq C\cdot\mus(B(x, t))$ for each $1\leq j\leq \# v$.
\end{lemma}
\begin{proof}
    Since $\mus(B(x,t))>0$ and $\mus$ is non-atomic, $U(x,t)\cap K\neq\emptyset$.
    From the weak separation condition, there exists some $\ell\in\N$ such that $\#\mathcal{S}_t(B(x,t))\leq\ell$ for any $x\in\R$ and $t>0$.
    By the invariant property of $\mus$ and since $\mus$ is a probability measure, we have
    \begin{align*}
        \mus(B(x, t)) &= \sum_{\sigma\in\Lambda_t(B(x, t))}p_\sigma\mus\circ S_\sigma^{-1}((B(x, t)))\leq \sum_{\sigma\in\Lambda_t(B(x, t))}p_\sigma\\
                          &= \sum_{S_\omega\in\mathcal{S}_t(B(x, t))}\sum_{\substack{\sigma\in\Lambda_t(B(x, t))\\S_\sigma=S_\omega}}p_\sigma.
    \end{align*}
    In particular, since $\#\mathcal{S}_t(B(x, t))\leq\ell$, get $\omega_0$ such that
    \begin{equation}\label{e:peqn}
        \sum_{\substack{\sigma\in\Lambda_t(B(x, t))\\S_\sigma=S_{\omega_0}}}p_\sigma\geq\mus(B(x, t))/\ell.
    \end{equation}
    Note that $S_{\omega_0}(K)\cap B(x, t)\neq\emptyset$, so that $S_{\omega_0}([0,1])\subseteq B(x,2 t)$.
    If $r_{\omega_0}<0$, get $k\in\mathcal{I}$ with $r_k<0$ and set $\omega_1=\omega_0k$; otherwise, take $\omega_1=\omega_0$.
    Now, let $\Delta_0\in\mathcal{F}_{s_0}$ be such that $\#\vs^c(\Delta_0)$ is maximal.
    Exactly as argued in \cref{p:wsc-es}, $\Delta_1\coloneqq S_{\omega_1}(\Delta_0)$ is a net interval in generation $r_{\omega_1}\cdot  s_0$ with $\vs(\Delta_1)=\vs(\Delta_0)$.
    Moreover, we know that if $\sigma$ generates some neighbour $f$ of $\Delta_0$, then $\omega_1\sigma$ generates the same neighbour $f$ of $\Delta_1$.
    Fix some $1\leq j\leq \#\vs(\Delta_1)$ and let $f_j$ be the neighbour of $\Delta_1$ corresponding to the index $j$.
    We then have by using \cref{e:peqn} and the above observation that
    \begin{align*}
        (\muv(\Delta_1))_j &= \mu(f_j^{-1}((0,1)))\sum_{\substack{\sigma\in\Lambda_{ s_0 r_{\omega_1}}\\\sigma\text{ generates }f_j}}p_\sigma\\
                           &\geq p_k\Bigl(\sum_{\substack{\sigma\in\Lambda_t(B(x, t))\\S_\sigma=S_{\omega_1}}}p_\sigma\Bigr)\cdot\mu(f_j^{-1}((0,1)))\cdot\sum_{\substack{\sigma\in\Lambda_{ s_0}\\\sigma\text{ generates }f_j}}p_\sigma\\
                           &\geq \mus(B(x, t))\cdot\frac{p_k\cdot (\muv(\Delta_0))_j}{\ell}\geq\mus(B(x, t))\cdot C_1
    \end{align*}
    where $C_1\coloneqq p_k\cdot \min_j(\muv(\Delta_0))_j/\ell$, which depends only on the IFS and choice of probabilities.

    Now let $\eta$ be any fixed path from $\vs(\Delta_0)$ to $v$ and let $\epsilon$ be the smallest strictly positive entry of $T(\eta)$.
    Let $\Delta$ be the unique net interval with symbolic $\gamma\eta$ where $\gamma$ is the symbolic representation of $\Delta_0$.
    Since $T(\eta)$ is non-negative and $\muv(\Delta)=\muv(\Delta_1)T(\eta)$ is a positive vector, we have that $(\muv(\Delta))_j\geq \mus(B(x, t))\cdot C_1\cdot \epsilon$.
    Taking $C\coloneqq C_1\epsilon$, we see that $C$ satisfies the requirements.
    Moreover, since $\Delta_0\in\mathcal{F}_{r_{\omega_0} s_0}$, taking $c= s_0L(\eta)\cdot r_{\min}^2$ and noting that $ t\cdot r_{\min}\leq |r_{\omega_0}|\leq t$,  we have that $\Delta\in\mathcal{F}_s$ where $ s\geq c t$.
    Finally, $\Delta\subseteq\Delta_1\subseteq S_{\omega_0}([0,1])\subseteq B(x,2 t)$ as required.
\end{proof}
\subsection{Measure properties of the essential class}
As our first consequence of this lemma, we establish that the interior essential points form a large subset of $K$.
\begin{theorem}\label{t:ess-meas}
    Let $\{S_i\}_{i\in\mathcal{I}}$ be an IFS satisfying the weak separation condition with attractor $K$ and let $v\in V(\mathcal{G}_{\ess})$ be arbitrary.
    Let
    \begin{equation*}
        E=\bigcup_{\substack{\Delta\in\mathcal{F}\\\vs(\Delta)=v}}\Delta\cap K.
    \end{equation*}
    Then if $\mus$ is any associated self-similar measure, $\mus(K\setminus E)=0$.
    In particular, $\mus(K\setminus K_{\ess})=0$.
\end{theorem}
\begin{proof}
    By \cref{l:bmeas}, there exist constants $c,C>0$ such that for any $t>0$ and ball $B(x,t)$ with $\mus(B(x,t))>0$, there exists some net interval $\Delta\in\mathcal{F}$ with $\Delta\subseteq B(x,2t)$, $\vs(\Delta)=v$, and $\mus(\Delta)\geq C\mus(B(x,r))$.
    We will construct a nested family of sets $E_1\supseteq E_2\supseteq\cdots$ such that each $E_n$ is a finite union of intervals, $\mus(E_n)\leq (1-C/3)^n$, and $K\setminus E\subseteq\bigcap_{n=1}^\infty E_n$.
    From this, the result clearly follows.

    First consider the ball $B_1=B(0,1)$.
    Get $\Delta_1\subseteq B(0,2)$ with $\vs(\Delta_1)=v$, set $E_1=[0,1]\setminus\Delta_1$ so that $\mus(E_1)\leq 1-C\leq 1-C/3$.
    Since $\Delta_1$ is an interval, $E_1$ is a finite union of intervals and clearly $K\setminus E\subseteq E_1$.
    Inductively, suppose $E_n$ is a finite union of intervals with $\mus(E_n)\leq(1-\lambda)^n$.
    Since each $E_n$ is a finite union of intervals, there is a family of balls $\{B(x_i,t_i)\}_{i=1}^m$ such that the balls only overlap pairwise on endpoints, $E_n=\bigcup_{i=1}^m B(x_i,t_i)$, and for any distinct $i_1,i_2,i_3$,
    \begin{equation}\label{e:bsep}
        B(x_{i_1},2t_{i_1})\cap B(x_{i_2},2t_{i_2})\cap B(x_{i_3},2t_{i_3})
    \end{equation}
    is either a singleton or the empty set and hence has measure $0$, as $\mus$ has no atoms.
    Now for each $1\leq i\leq m$, apply \cref{l:bmeas} to get $\Delta_n^{i}\subseteq B(x_i,2t_i)$ with $\mus(\Delta_n^{i})\geq C\mus(B(x_i,t_i))$.
    While the $\Delta_n^{i}$ need not be disjoint, by \cref{e:bsep}, there exists a sub-collection labelled without loss of generality $\{\Delta_n^{i}\}_{i=1}^{m'}$ such that
    \begin{equation*}
        \sum_{i=1}^{m'} \mus(\Delta_n^{i})\geq\frac{1}{3}\sum_{i=1}^{m} \mus(\Delta_n^{i})
    \end{equation*}
    and $\Delta_n^{i}\cap\Delta_n^{j}$ is at most a singleton for $i\neq j$.
    (To do this, pick the interval $\Delta_n^{i}$ with the largest measure and remove any net intervals $\Delta_n^{j}$ where $\Delta_n^{j}\cap\Delta_n^{i}$ is not a singleton.
    By \cref{e:bsep} and the geometry in $\R$, there are at most 2 such indices $j$.
    Then repeat until the set is exhausted.)

    Set $E_{n+1}=E_n\setminus\bigcup_{i=1}^{m'}\Delta_n^{i}$.
    Each $\Delta_n^{i}$ is an interval with $\vs(\Delta_n^{i})=v$, so that $E_{n+1}$ is a finite union of intervals with $K\setminus E\subseteq E_{n+1}$, and
    \begin{align*}
        \mus(E_{n+1}) &= \mus(E_n)-\sum_{i=1}^{m'}\mus(\Delta_n^{i})\leq \mus(E_n)-\frac{C}{3}\sum_{i=1}^{m} \mus(B(x_i,t))\\
                      &\leq (1-C/3)\mus(E_n)\leq (1-C/3)^{n+1}
    \end{align*}
    as claimed.
\end{proof}
\begin{remark}
    It can also be shown, using similar techniques, that if $s=\dimH K$, then $\mathcal{H}^s(K\setminus K_{\ess})=0$ where $\mathcal{H}^s$ is the $s$-dimensional Hausdorff measure.
    This follows from Ahlfors regularity of self-similar sets under the weak separation condition~\cite[Thm.~2.1]{fhor2015} along with \cref{l:nbdet}, in place of \cref{l:bmeas}.
\end{remark}

\subsection{Local dimensions and periodic points}
The notion of a periodic point was introduced by Hare, Hare and Matthews for IFS of the form $\{x\mapsto rx+d_i\}_{i\in\mathcal{I}}$ with $0<r<1$ satisfying the finite type condition~\cite{hhm2016}.
In this section, we take advantage of the general matrix product formula, \cref{t:approx}, to establish symbolic formulas for the local dimensions at certain points which we call periodic.
\begin{definition}
    Given a Borel probability measure $\mu$, by the \defn{lower local dimension} of $\mu$ at $x\in\supp\mu$, we mean the number
    \begin{equation*}
        \dimll\mu(x)=\liminf_{t\downarrow 0}\frac{\log \mu(B(x,t))}{\log t}.
    \end{equation*}
    The \defn{upper local dimension} is defined analogously; when the upper and lower local dimensions coincide, we call the shared value the \defn{local dimension} of $\mu$ at $x$, denoted by $\diml\mu(x)$.
\end{definition}
\begin{definition}
    A \defn{periodic point} is a point $x\in K$ where every symbolic representation of $x$ is of the form
    \begin{equation*}
        [x]=(e_1,\ldots,e_n,\theta,\theta,\ldots)
    \end{equation*} 
    where $n$ is minimal and $\theta=(\theta_1,\ldots,\theta_m)$ is a cycle of $\mathcal{G}$ with minimal length.
    In this case, we call $\theta$ a \defn{period} of the symbolic representation.
\end{definition}
Intuitively, periodic points are the natural analogue of the rational numbers; for example, with respect to the IFS $\{x\mapsto x/2,x\mapsto x/2+1/2\}$, the periodic points of this IFS are precisely the rational numbers in $[0,1]$.
Under the weak separation condition, it is straightforward to see that the periodic points form a countable dense subset of $K$: if $x,y\in K$ have symbolic representations of the form $\gamma\eta_1$ and $\gamma\eta_2$, then both $x$ and $y$ are in the net interval with symbolic representation $\gamma$.

The proofs of \cref{l:ldapprox} and \cref{p:period} are motivated by the proofs~\cite[Thm. 2.6 and Prop. 2.7]{hhn2018}.

Fix some $x\in K$.
Enumerate $\{h_j:j=1,\ldots,n\}=\{S_\sigma(0),S_\sigma(1):\sigma\in\Lambda_t\}$ with $h_1<\cdots<h_n$.
If $x\neq h_j$ for each $1\leq j\leq n$, then there is a unique net interval $\Delta_t(x)=[h_i,h_{i+1}]$ of generation $t$ containing $x$.
We then say $\Delta_t^-(x)$ is the empty set if $i=1$ or $(h_{i-1},h_i)\cap K=\emptyset$, and $\Delta_t^-(x)=[h_{i-1},h_i]$ otherwise, and we define $\Delta_t^+$ similarly.
Then set
\begin{align*}
    M_t(x)=\Delta_t^-(x)\cup\Delta_t(x)\cup\Delta_t^+(x).
\end{align*}
Otherwise, $x=h_m$ for some $m$, and we write $\Delta^{1}_t(x)=[h_{m-1},h_m]$ if $m\neq 1$ and $(h_{m-1},h_m)\cap K$ is non-empty, and similarly for $\Delta^{2}_t(x)$, and set
\begin{equation*}
    M_t(x)=\Delta^{1}_t(x)\cup\Delta^{2}_t(x).
\end{equation*}
We have the following basic estimation:
\begin{lemma}\label{l:ldapprox}
    Let $\{S_i\}_{i\in\mathcal{I}}$ be an IFS as in \cref{e:ifs} and let $x\in K$ be such that $\sup\{\lm(\Delta):x\in\Delta,\Delta\in\mathcal{F}\}<\infty$.
    Then if $\mus$ is any associated self-similar measure,
    \begin{equation*}
        \diml\mus(x)=\lim_{t\to 0}\frac{\log\mu_{\bm{p}}(M_t(x))}{\log t}
    \end{equation*}
    provided the limit on the right exists.
    Similar statements hold with respect to the limit supremum and limit infimum for the upper and lower local dimensions respectively.
\end{lemma}
\begin{proof}
    Suppose the local dimension exists and equals $D$.
    Recall that if $\Delta\in\mathcal{F}_t$, then $t\geq\tg(\Delta)=\lm(\Delta)\diam(\Delta)$.
    Thus there exists some constant $0<\epsilon$ such that for any $t>0$ and $\Delta\in\mathcal{F}_t$ with $x\in\Delta$, $\epsilon t<\diam(\Delta)$.
    Moreover, $\diam(\Delta)\leq t$ always holds by the net interval construction.

    If $x$ is a boundary point, get $s$ such that $x$ is an endpoint of $\Delta_s(x)$ and
    \begin{equation*}
        B(x,\epsilon s)\subseteq\Delta_s^{1}(x)\cup\Delta_s^{2}(x)\subseteq B(x,2s)
    \end{equation*}
    where the notation is as above.
    Otherwise if $x$ is not a boundary point, then
    \begin{equation*}
        B(x,\epsilon s)\subseteq\Delta^-_s(x)\cup\Delta_s(x)\cup\Delta_s^+(x)\subseteq B(x,2s)
    \end{equation*}
    In either case, $B(x,\epsilon s)\subseteq M_s(x)\subseteq B(x,2s)$ so that
    \begin{align*}
        \left(\frac{\log\epsilon+\log s}{\log s}\right)&\left(\frac{\log \mus(B(x,\epsilon s))}{\log \epsilon s}\right)\\
        &\leq \frac{\log \mus(M_s(x))}{\log s}\leq\left(\frac{\log s+\log 2}{\log s}\right)\left(\frac{\log\mus(B(x,2s))}{\log 2s}\right).
    \end{align*}
    The limit of the left and right both exist and are equal to $D$; hence, the limit of the middle expression exists and equals $D$.
    The arguments for the upper and lower dimension follow similarly.
\end{proof}
In the following proposition, recall that for a path $\theta$, $L(\theta)$ is the length of the path defined in \cref{d:e-len}.
\begin{proposition}\label{p:period}
    Let $\{S_i\}_{i\in\mathcal{I}}$ be any IFS and suppose $x$ is a periodic point with period $\theta=(e_1,\ldots,e_s)$.
    Then the local dimension of $\mu$ at $x$ exists and is given by
    \begin{equation*}
        \dim_{loc}\mu(x)=\frac{\log\spr(T(\theta))}{\log L(\theta)}
    \end{equation*}
    where if $x$ is a boundary point of a net interval with two different symbolic representations given by periods $\theta$ and $\phi$, then $\theta$ is chosen to satisfy
    \begin{equation*}
        \frac{\log\spr(T(\theta))}{\log L(\theta)}\leq \frac{\log\spr(T(\phi))}{\log L(\phi)}.
    \end{equation*}
\end{proposition}
\begin{proof}
    First, suppose $x$ is a periodic point with two distinct symbolic representations with periods $\theta=(\theta_1,\ldots,\theta_\ell)$ and $\phi=(\phi_1,\ldots,\phi_{\ell'})$, so that $x$ is an endpoint of some net interval $\Delta\in\mathcal{F}$.
    We first note that
    \begin{align*}
        \mu_{\bm{p}}(\Delta_t^{1}(x))&=\bigl\lVert T(e_1,\ldots,e_{j},\underbrace{\theta,\ldots,\theta}_{m},\theta_1,\ldots,\theta_t)\bigr\rVert\\
        \mu_{\bm{p}}(\Delta_t^{2}(x))&=\bigl\lVert T(e_1',\ldots,e'_{j'},\underbrace{\phi,\ldots,\phi}_{m'},\phi_1,\ldots,\phi_{t'})\bigr\rVert
    \end{align*}
    for $t$ sufficiently small, $t < \ell$, and $t'<\ell'$.
    Now, get constants $c_i$ which do not depend on $t$ such that
    \begin{align}\label{e:comp}
        \norm{(T(\theta))^{m+1}} &\leq \bigl\lVert T(\underbrace{\theta,\ldots,\theta}_m,\theta_1,\ldots,\theta_t)\bigr\rVert\cdot\norm{T(\theta_{t+1},\ldots,\theta_\ell)}\nonumber\\
                                 &\leq c_1\bigl\lVert T(e_1,\ldots,e_{j},\underbrace{\theta,\ldots,\theta}_m,\theta_1,\ldots,\theta_t)\bigr\rVert\leq c_2\norm{T(\theta)^m}.
    \end{align}
    Moreover, since
    \begin{equation*}
        L(e_1,\ldots,e_j)L(\theta)^m L(\theta_1,\ldots,\theta_t) r_{\min}\leq t\leq L(e_1,\ldots,e_j)L(\theta)^m L(\theta_1,\ldots,\theta_t),
    \end{equation*}
    we have $L(\theta)^m\asymp t$ with constants of comparability not depending on $t$.
    Thus, there exist $k_i$ not depending on $t$ so that
    \begin{equation*}
        \frac{\log k_1\norm{T(\theta)^{m+1}}^{1/(m+1)}}{\log k_3\cdot L(\theta)}\geq\frac{\log \mu(\Delta_t^{1}(x))}{\log t}\geq \frac{\log k_2\norm{(T(\theta))^m}^{1/m}}{\log k_4\cdot L(\theta)}
    \end{equation*}
    and taking the limit as $t$ goes to 0 yields
    \begin{equation*}
        \lim_{t\to 0}\frac{\log \mu_{\bm{p}}(\Delta_t^{1}(x))}{\log t} = \frac{\log \spr(T(\theta))}{\log L(\theta)}
    \end{equation*}
    In the exact same way, we get
    \begin{equation*}
        \lim_{t\to 0}\frac{\log \mu_{\bm{p}}(\Delta_t^{2}(x))}{\log t} = \frac{\log \spr(T(\phi))}{\log L(\phi)}.
    \end{equation*}
    Now, since $x$ is a periodic point, the set $\{\vs(\Delta):x\in\Delta,\Delta\in\mathcal{F}\}$ is finite.
    Since $\lm(\Delta)$ depends only on $\vs(\Delta)$, $\sup\{\lm(\Delta):x\in\Delta,\Delta\in\mathcal{F}\}<\infty$ and the assumptions for \cref{l:ldapprox} hold.
    Then by the power mean inequality, we have
    \begin{align*}
        \diml\mu_{\bm{p}}(x) &= \lim_{t\to 0}\frac{\log\mu_{\bm{p}}(\Delta_t^{1}(x))+\mu_{\bm{p}}(\Delta^{2}_t(x))}{\log t}\\
                                   &= \min\left(\lim_{t\to 0}\frac{\log\mu_{\bm{p}}(\Delta_t^{1}(x))}{\log t},\lim_{t\to 0}\frac{\log\mu_{\bm{p}}(\Delta^{2}_t(x))}{\log t}\right)\\
                                   &= \min\left(\lim_{t\to 0}\frac{\log\spr T(\theta)}{\log L(\theta)},\lim_{t\to 0}\frac{\log\spr T(\phi)}{\log L(\phi)}\right)
    \end{align*}
    since the final two limits in the maximum exist, as claimed.

    If $x$ is an endpoint of some net interval but has only one symbolic representation, then either $\Delta_t^{1}(x)$ or $\Delta_t^{2}(x)$ is empty for sufficiently small $t$ and the argument is identical, but easier.

    Finally, suppose $x$ is not an endpoint of any net interval, and thus has unique symbolic representation $[x] = (e_1,\ldots,e_{j},\theta,\theta,\ldots)$ where $\theta=(\theta_1,\ldots,\theta_\ell)$.
    In this situation, $\Delta_1$ has symbolic representation $(e_1,\ldots,e_{j},\theta^n)$ and $\Delta_2$ has symbolic representation $(e_1,\ldots,e_{j},\theta^{n+1})$ for any $n\in\N$, we have $\Delta_2\subseteq\Delta_1^\circ$.
    Thus for any $t$ sufficiently small, there exists some $m\in\N$, such that $\Delta_1\subseteq\Delta_t(x)\subseteq M_t(x)\subseteq\Delta_2$ where $\Delta_1$ has symbolic representation $(e_1,\ldots,e_j,\theta^m)$ and $\Delta_2$ has symbolic representation $(e_1,\ldots,e_j,\theta^{m+2})$.
    Similarly as argued in \cref{e:comp}, there exist constants $c_1,c_2$ such that $\norm{T(\theta)^{m+2}}\leq c_1\mu(\Delta_t(x))\leq c_2\norm{T(\theta)^m}$.
    In addition, since $M_t(x)\subseteq\Delta_1$, we have $\mu(M_t(x))\leq\mu(\Delta_1)$ and there exist constants $c_1',c_2'$ such that $\norm{T(\theta)^{m+2}}\leq c_1'\mu(M_t(x))\leq c_2'\norm{T(\theta)^m}$.

    The argument proceeds identically as before.
\end{proof}
\section{Multifractal formalism under the weak separation condition}
In this section, we prove the multifractal formalism results under the weak separation condition.
\subsection{Density of local dimensions at periodic points}
We first show that under the weak separation condition periodic points are abundant, in that the set of local dimensions at periodic points is dense in the set of local dimensions in the essential class.
This generalizes a result of Hare, Hare and Ng on local dimensions~\cite[Cor.~3.15]{hhn2018} for IFSs satisfying substantially stricter conditions.
This property can be useful in computing the exact set of possible local dimensions; see, for example, \cref{sss:non-comm-dim} or the discussions of examples in~\cite{hhm2016,hhn2018,hhs2018}.
\begin{theorem}\label{t:per-dens}
    Let $\{S_i\}_{i\in\mathcal{I}}$ be an IFS satisfying the weak separation condition and $\mus$ an associated self-similar measure.
    Then the set of local dimensions at periodic points is dense in $\{\dimul(x):x\in K_{\ess}\}$ and $\{\dimll(x):x\in K_{\ess}\}$.
\end{theorem}
\begin{proof}
    Let $x$ be an interior essential point.
    Either there exists some $s_0$ such that there is a unique essential net interval $\Delta_0\in\mathcal{F}_{s_0}$ containing $x$, or there exists essential net intervals $\Delta_0^{1},\Delta_0^{2}$ such that $\{x\}=\Delta_0^{1}\cap\Delta_0^{1}$.
    The cases are similar, but the latter is slightly harder, so we treat that here.

    Let $t_0>0$ be such that $B(x,2t_0)\subseteq\Delta_0^{1}\cup\Delta_0^{2}$.
    Arguing similarly to \cref{l:bmeas}, there exists constants $c,C>0$ such that for any $0<t\leq t_0$, there exists $\Delta_t^{1}\subseteq\Delta_t^{2}\subseteq B(x,2t)$ and for each $k=1,2$, we have $\Delta_t^{k}\in\mathcal{F}_s$ where $t\geq s\geq ct$,
    \begin{equation*}
        \min\{\muv(\Delta_t^{k})_j:1\leq j\leq \# \vs(\Delta_t^{k})\}\geq C\mus(B(x,t)),
    \end{equation*}
    and $\vs(\Delta_i^{k})=\vs(\Delta_0^{k})$.
    We may also assume that $\Delta_t^{1}$ and $\Delta_t^{2}$ do not contain $x$ as an endpoint.
    In particular, for each $0<t\leq t_0$, there exists some $k\in\{1,2\}$ such that $\Delta_t^{k}\subseteq(\Delta_0^{k})^\circ$.
    Set $\Delta_t=\Delta_t^{k}$ and let $\eta_t$ be the path in the transition graph corresponding to $\Delta_t^{k}\subseteq\Delta_0^{k}$, which is a cycle since the two net intervals have the same neighbour set.
    Let $\gamma_1$ be the symbolic representation of $\Delta_0^{1}$ and $\gamma_2$ the symbolic representation of $\gamma_0^{2}$

    For each $0<t\leq t_0$, let $x_t$ be any periodic point with period $\eta_t$.
    We note that since $x_t$ is not the boundary point of any net interval, we have by \cref{p:period}
    \begin{equation*}
        \diml\mus(x_t)=\frac{\log\spr T(\eta_t)}{\log L(\eta_t)}.
    \end{equation*}

    Fix $t$ as above, and let $\Delta_0\in\{\Delta_0^{1},\Delta_0^{2}\}$ be such that $x_0\in\Delta_0^\circ$.
    Let $\Delta_0$ have symbolic representation $\gamma$.
    By definition of $c$, we observe that $t\geq \tg(\Delta_t)\geq c r_{\min}t$.
    Since $\tg(\Delta_t)=L(\gamma)L(\eta_t)$, there exist constants $c_1,c_2>0$ (not depending on $t$) such that
    \begin{equation*}
        c_1 2t\leq L(\eta_t)\leq c_2 t.
    \end{equation*}

    We also bound $\spr T(\eta_t)$.
    Since $\Delta_t\subseteq B(x,2t)$ has symbolic representation $\gamma\eta_t$, we have $\norm{T(\gamma\eta_t)}\leq \mus(B(x,2t))$ and since $T(\gamma)$ is a transition matrix, there exists some $C_1>0$ such that
    \begin{equation*}
        \spr T(\eta_t)\leq \norm{T(\eta_t)}\leq C_1\mus(B(x,2t))
    \end{equation*}
    (just take $C_1$ to be the smallest strictly positive entry of $T(\gamma_1)$ and $T(\gamma_2)$).
    On the other hand, since $\muv(\Delta_t)=\muv(\Delta_0)T(\eta_t)$, we have
    \begin{align*}
        \spr T(\eta_i)&\geq \frac{\min\{\muv(\Delta_t)_j:1\leq j\leq \# v\}}{\max\{\muv(\Delta_0)_j:1\leq j\leq\# v\}}\geq\frac{C\mus(B(x,t))}{\max\{\muv(\Delta_0^{k})_j:1\leq j\leq\# v,1\leq k\leq 2\}}\\
                      &= C_2 \mus(B(x,t)).
    \end{align*}

    To summarize, we have shown that
    \begin{align*}
        \frac{\log C_2+\log \mus(B(x,t))}{\log c_2+\log t}&\geq \diml\mus(x_t)=\frac{\log \spr T(\eta_t)}{\log L(\eta_t)}\\
                                                          &\geq\frac{\log C_1+\log \mus(B(x,2t))}{\log c_1+\log 2t}.
    \end{align*}

    Let $\alpha=\dimul\mus(x)$ and let $\epsilon>0$ be arbitrary.
    Get some $t_1>0$ such that for all $0<t\leq t_1$,
    \begin{equation*}
        \frac{\log C_2+\log\mus(B(x,t))}{\log c_2 +\log t}\leq\alpha+\epsilon
    \end{equation*}
    and then choose $0<t\leq\min\{t_0,t_1\}$ such that
    \begin{equation*}
        \frac{\log C_1+\log\mus(B(x,2t))}{\log c_1+\log 2t}\geq\alpha-\epsilon.
    \end{equation*}
    Since $\epsilon>0$ was arbitrary, it follows that the set of local dimensions at periodic points is dense in $\{\dimul(x):x\in K_{\ess}\}$.
    The result for lower local dimensions holds identically.
\end{proof}

\subsection{The \texorpdfstring{$L^q$}{Lq}-spectrum, dimension spectrum, and multifractal formalism}
In this section, we show how to extend a result of Feng and Lau~\cite{fl2009} to hold with respect to a larger, more natural class of intervals.

Let $\mu$ be a compactly supported finite Borel measure and let $V\subseteq\R$ be any open set with $\mu(V)>0$.
Then the \defn{$L^q$-spectrum of $\mu$ on $V$}, denoted by $\tau_V(\mu,q)$, is given by
\begin{equation*}
    \tau_V(\mu,q)=\liminf_{t\downarrow 0}\frac{\log\sup \sum_i\mu\bigl(B(x_i,t)\bigr)^q}{\log t}
\end{equation*}
where the supremum is over families of disjoint closed balls $\{B(x_i,t)\}_i$ with $x_i\in \supp\mu$ and $B(x_i,t)\subseteq V$.
A direct application of Hölder's inequality shows that $\tau_V(q)$ is a concave function.
When $V=\R$, we write $\tau(\mu,q)=\tau_{\R}(\mu,q)$.

Since $\tau_V(\mu,q)$ is a concave function in $q$, its \defn{concave conjugate} is given by
\begin{equation*}
    \tau_V^*(\mu,\alpha)\coloneqq \inf\{\alpha q-\tau_V(q):q\in\R\}.
\end{equation*}
We set
\begin{align*}
    D_V(\mu) &=\{\alpha\in\R:\diml\mu(x)=\alpha\text{ for some }x\in K\cap V\}
\end{align*}
and
\begin{equation*}
    K_V(\mu,\alpha)=\{x\in K\cap V:\diml\mu(x)=\alpha\}.
\end{equation*}
Understanding the geometric properties of the sets $K_V(\mu,\alpha)$ is a natural way to understand the structure of $\mu$.

A heuristic relationship between the values if $\dimH K_V(\mu,\alpha)$ and the concave conjugate of $L^q$-spectrum, known as the multifractal formalism, has been studied by many authors (see, for example, \cite{cm1992,fen2003a,fen2009,fl2009,flw2005,hjk+1986,lau1995,ln1999,pat1997,pw1997,shm2005}).
\begin{definition}\label{d:cmf}
    Let $\mu$ be a compactly supported finite Borel measure and let $V\subset\R$ have $\mu(V)>0$.
    We say that the measure $\mu$ satisfies the \defn{complete multifractal formalism} with respect to $V$ if
    \begin{enumerate}[nl,r]
        \item $D_V(\mu)=[\alpha_{\min},\alpha_{\max}]$ where
            \begin{align*}
                \alpha_{\min}&=\lim_{q\to +\infty}\frac{\tau_V(q)}{q} & \alpha_{\max}&=\lim_{q\to -\infty}\frac{\tau_V(q)}{q}.
            \end{align*}
        \item For any $\alpha\in[\alpha_{\min},\alpha_{\max}]$, $\tau_V^*(\alpha)=\dimH K_V(\alpha)$.
    \end{enumerate}
\end{definition}
Note that we do not comment on differentiability of $\tau_V(q)$.
\subsection{Weak regularity and restricting the \texorpdfstring{$L^q$}{Lq}-spectrum}\label{ss:weak-reg}
We now begin the setup for the statement and proof of \cref{ti:multi-wsc}.
In the statement, we are restricting our measure $\mus$ to a set $K\cap E$ where $E$ is a finite union of closed intervals.
In the interior of $E$, this does not cause any problems: in general, if $V$ is any open set, then $\tau_V(\mus,q)\geq\tau(\mus,q)$.
However, the measure of balls centred at the endpoint of a closed interval could be substantially smaller.

For example, suppose $\mus$ is the uniform Cantor measure (corresponding to the IFS $S_1(x)=x/3$ and $S_2(x)=x/3+2/3$ with probabilities $p_1=p_2=1/2$) and $x$ is the point with symbolic representation consisting of increasingly long alternating stretches of $1$s and $2$s.
Then, the one-sided upper local dimensions of $\mus|_{[0,x]}$ at $x$ is not equal to the everywhere constant value of the local dimension of $\mus$.

In this section, we introduce the notion of weak regularity, which ensures that this situation does not happen.
We also prove some results which show that this hypothesis is not too challenging to satisfy in general.

We recall that $E$ is \defn{Ahlfors regular} if there is some $s>0$ and $a,b>0$ such that
\begin{equation*}
    a t^s \leq \mathcal{H}^s(E\cap B(x,t))\leq bt^s
\end{equation*}
for all $x\in E$ and $t$ sufficiently small.
If $K$ is the attractor of an IFS satisfying the weak separation condition, then $K$ is always Ahlfors regular (see, for example, \cite{fhor2015}).
\begin{definition}\label{d:weak-reg}
    We say that a set $E$ is \defn{weakly regular} if there is some $\epsilon>0$ such that for all $t>0$ sufficiently small,
    \begin{equation*}
        E\cap \bigl(B(x,t)\setminus B(x,\epsilon t)\bigr)\neq\emptyset
    \end{equation*}
    for all $x\in E$.
\end{definition}
We begin with the following useful observation.
\begin{lemma}\label{l:boundary-reg}
    Suppose $K$ is Ahlfors regular and $E\subset K$ is compact.
    Then $E$ is weakly regular if and only if the boundary of $E$ (in the topology relative to $K$) is weakly regular.
\end{lemma}
\begin{proof}
    The forward direction is immediate.
    Conversely, let $0<\epsilon_0<1$ be the constant from weak regularity of the boundary of $E$.
    Suppose $x$ is in the interior of $E$ relative to $K$ and let $t>0$.
    If $B(x,t\epsilon_0/4)\cap K=B(x,t\epsilon_0/4)\cap E$, then for $\epsilon<\epsilon_0/4$,
    \begin{equation*}
        \mathcal{H}^s\bigl(E\cap B(x,t\epsilon_0/4)\setminus B(x,\epsilon t)\bigr)\geq(a(\epsilon_0/4)^s-\epsilon^s b)t^s>0
    \end{equation*}
    for some $\epsilon>0$ depending only on $a$, $b$, $s$, and $\epsilon_0$.
    Thus $E\cap B(x,t)\setminus B(x,\epsilon t)\neq\emptyset$.
    Otherwise, there is some $y\in B(x,t\epsilon_0/4)$ in the boundary of $E$ so that
    \begin{equation*}
        \emptyset\neq E\cap B(y,t/2)\setminus B(y,t\epsilon_0/2)\subset E\cap B(x,t)\setminus B(x,t\epsilon_0/4).
    \end{equation*}
    as required.
\end{proof}
The main point behind weak regularity is the following lemma.
\begin{lemma}\label{l:br-restriction}
    Let $\mu$ be a Borel probability measure with compact support $K$, and let $V$ be an open set with $\mu(V)>0$.
    Suppose $E\subset V$ is a finite union of closed intervals such that $E\cap K$ is weakly regular.
    Then $\tau_{V}(\mu,q)\leq\tau(\mu|_E,q)$.
\end{lemma}
\begin{proof}
    This follows directly for $q\geq 0$ since for all $t$ sufficiently small, $B(x,t)\subset V$ for any $x\in E\cap K$.

    Otherwise, let $q<0$ and let $t$ be sufficiently small such that each interval in $E$ has length at least $2t$ and $B(x,t)\subset V$ for any $x\in E\cap K$.
    Let $\{B(x_i,t)\}_i$ be an arbitrary centred packing of $E\cap K$.
    By weak regularity, there is some $\epsilon>0$ such that for each $i$, there is some $y_i\in E\cap K$ such that $B(y_i,\epsilon t)\subseteq B(x_i,t)\cap E$.
    Therefore,
    \begin{equation*}
        \sum_i\mu|_E(B(x_i,t))^q\leq\sum_i\mu|_E(B(y_i,\epsilon t))^q=\sum_i \mu(B(y_i,\epsilon t))^q.
    \end{equation*}
    But $\{B(x_i,t)\}_i$ was arbitrary, so the desired result follows.
\end{proof}
We now show that intervals $J$ with $J\cap K$ weakly regular are abundant.
Recall that $F^{(\delta)}$ denotes the (closed) $\delta$-neighbourhood of a set $F$.
\begin{lemma}\label{l:er-extension}
    Let $\{S_i\}_{i\in\mathcal{I}}$ be an IFS satisfying the weak separation condition with attractor $K$ and let $\delta>0$.
    Then if $F\subset K_{\ess}$ is any compact subset, there is a finite union of essential net intervals $E=\Delta_1\cup\cdots\cup\Delta_n$ such that $F\subset E\subset F^{(\delta)}$ and $E\cap K$ is weakly regular.
\end{lemma}
\begin{proof}
    For each $t>0$ set
    \begin{equation*}
        \mathcal{H}_t=\{\Delta\in\mathcal{F}_t:\vs(\Delta)\in V(\mathcal{G}_{\ess})\}
    \end{equation*}
    and let
    \begin{equation*}
        U_t=\Bigl(\bigcup_{\Delta\in\mathcal{H}_t}\Delta\Bigr)^\circ.
    \end{equation*}
    It follows directly from the definition that $K_{\ess}=\bigcup_{t>0}U_t$.
    We may assume $\delta$ is sufficiently small so that $F^{(\delta)}\subset K_{\ess}$.
    Since $F^{(\delta)}$ is compact, get $t_0$ such that $F^{(\delta)}\subset U_{t_0}$, let $t_1=\min\{t_0,\delta/2\}$, and set
    \begin{align*}
        \mathcal{E} &=\{\Delta\in\mathcal{H}_{t_1}:\Delta\cap F\neq\emptyset\},\\
        E_0&=\bigcup_{\Delta\in\mathcal{E}}\Delta.
    \end{align*}
    Note that $F\subset E_0\subset F^{(\delta/2)}$ since $\diam(\Delta)\leq t_1$ for any $\Delta\in\mathcal{F}_{t_1}$.
    Now for each $\Delta\in\mathcal{E}$, get $0<t\leq t_1$ and $\sigma,\tau\in\Lambda_t$ such that $r_\sigma,r_\tau>0$ and $\Delta\subseteq[S_\sigma(0),S_\tau(1)]\subseteq \Delta^{(\delta/2)}$.
    Finally, set
    \begin{equation*}
        \mathcal{E}_\Delta=\{\Delta'\in\mathcal{F}_t:\Delta'\subset[S_{\sigma}(0),S_\tau(1)]\}.
    \end{equation*}
    Observe that
    \begin{equation*}
        K\cap\bigcup_{\Delta'\in\mathcal{E}_\Delta}\Delta'=K\cap[S_\sigma(0),S_\tau(1)]\subset F^{(\delta)}
    \end{equation*}
    is weakly regular by \cref{l:boundary-reg}.
    Moreover, since $F^{(\delta)}\subset U_{t_0}$, each $\Delta'\in\mathcal{E}_\Delta$ is essential.
    Thus since a union of weakly regular sets is again weakly regular,
    \begin{equation*}
        E=\bigcup_{\Delta\in\mathcal{E}}\bigcup_{\Delta'\in\mathcal{E}_\Delta}\Delta
    \end{equation*}
    satisfies the requirements.
\end{proof}
We conclude this section with the following observation.
\begin{lemma}\label{l:equi-reg}
    Let $\{S_i\}_{i\in\mathcal{I}}$ be any equicontractive IFS satisfying the weak separation condition.
    If $E$ is any finite union of net intervals such that $E\cap K$ contains no isolated points, then $E$ is weakly regular.
\end{lemma}
\begin{proof}
    Write $S_i(x)=r x+d_i$ where $0<r<1$.
    It suffices to prove that $[0,x]\cap K$ and $[x,1]\cap K$ are weakly regular for any $x=S_\sigma(z)$ where $\sigma\in\mathcal{I}^*$ and $z\in\{0,1\}$, where $x$ is not an isolated point of $[0,x]\cap K$ or $[x,1]\cap K$.
    We will prove the case $[0,S_\sigma(0)]\cap K$; the remaining cases are either analogous or easier.

    Suppose for contradiction $[0,S_\sigma(0)]$ is not weakly regular and get indices $(k_n)_{n=1}^\infty$ and a sequence $(\epsilon_n)_{n=1}^\infty$ converging monotonically to zero such that
    \begin{equation*}
        [S_\sigma(0)-r^{k_n},S_\sigma(0)-\epsilon_n r^{k_n})\cap K=\emptyset
    \end{equation*}
    for each $n\in\N$.
    Since $S_\sigma(0)$ is an accumulation point from the right, there is some $\tau_n\in\mathcal{I}^{k_n}$ such that $S_{\tau_n}([0,1])\supseteq[S_\sigma(0)-\delta,S_\sigma(0)]$ for some $\delta>0$ sufficiently small.
    But $S_{\tau_n}(\{0,1\})\cap[S_\sigma(0)-r^{k_n},S_\sigma(0)-\epsilon_n r^{k_n})=\emptyset$, which forces
    \begin{equation*}
        |S_{\tau_n}(0)-S_{\sigma_n}(0)|\leq\epsilon_n r^{k_n}.
    \end{equation*}
    where $\sigma_n\in\mathcal{I}^{k_n}$ is the word with $\sigma$ as a prefix and $S_{\sigma_n}(0)=S_\sigma(0)$.
    This contradicts the weak separation condition by \cite[Thm.~1]{zer1996}.
\end{proof}
\begin{remark}
    In the general case, the same argument gives that $x\mapsto\lambda x$ for some $\lambda\neq 0$ is an accumulation point of $\{S_\sigma^{-1}\circ S_\tau:\sigma,\tau\in\mathcal{I}^*\}$ in the topology of pointwise convergence.
    The equicontractive assumption gives that $\lambda=1$, but it is unclear how to guarantee this in general.
\end{remark}
\subsection{Multifractal formalism for the essential class}
We now prove the multifractal formalism for the essential class.
We begin with the following result, which is contained in \cite[Thm.~5.4]{fl2009}:
\begin{proposition}[\cite{fl2009}]\label{p:multi-wsc}
    Let $\{S_i\}_{i\in\mathcal{I}}$ be an IFS satisfying the weak separation condition and let $\mus$ be an associated self-similar measure.
    Let $U_0$ be any open ball which attains the maximal value in \cref{e:wsc-max}.
    Then $\mus$ satisfies the complete multifractal formalism with respect to $U_0$.
\end{proposition}
Using the notion of the essential class, we can obtain a strictly stronger extension of this proposition.
We first note the following straightforward lemma:
\begin{lemma}[\cite{fl2009}]\label{l:u0-sig}
    Let $\{S_i\}_{i\in\mathcal{I}}$ be an IFS satisfying the weak separation condition and let $\mus$ be an associated self-similar measure.
    Let $U_0$ be any open ball which attains the maximal value in \cref{e:wsc-max}.
    Then if $\sigma\in\mathcal{I}^*$ is arbitrary,
    \begin{enumerate}[nl,r]
        \item $\tau_{S_\sigma(U_0)}(\mus,q)=\tau_{U_0}(q)$,
        \item $D_{S_\sigma(U_0)}(\mus)=D_{U_0}(\mus)$, and
        \item $\dimH K_{U_0}(\mus,\alpha)=\dimH K_{S_\sigma(U_0)}(\mus,\alpha)$.
    \end{enumerate}
\end{lemma}
\begin{proof}
    Statement (i) is~\cite[Cor.~5.6]{fl2009}.
    Statements (ii) and (iii) are implicit in the usage of~\cite[Lem.~2.5]{fl2009}.
\end{proof}
We obtain the following extension of \cref{p:multi-wsc}.
In light of \cref{p:mx-wsc} and \cref{l:er-extension}, our result is strictly stronger.
\begin{theorem}\label{t:multi-wsc}
    Let $\{S_i\}_{i\in\mathcal{I}}$ be an IFS satisfying the weak separation condition and let $\mus$ be a self-similar measure.
    Let $\Delta_1,\ldots,\Delta_n$ be any essential net intervals such that with $E\coloneqq \Delta_1\cup\cdots\cup\Delta_n$, $E\cap K$ is weakly regular.
    Then with $\nu=\mus|_{E}$,
    \begin{enumerate}[nl,r]
        \item $\nu$ satisfies the complete multifractal formalism,
        \item the set
            \begin{equation*}
                P(\mus)\coloneqq \{\diml\mus(x):x\in K_{\ess},x\text{ periodic}\}
            \end{equation*}
            is dense in $D(\nu)$, and
        \item the sets of local dimensions satisfy
            \begin{align*}
                D(\nu) &= \{\diml\mus(x):x\in K_{\ess},\diml\mus(x)\text{ exists}\}\\
                       &= \{\dimll\mus(x):x\in K_{\ess}\}= \{\dimul\mus(x):x\in K_{\ess}\}.
            \end{align*}
    \end{enumerate}
    Moreover, the values of $\tau(\nu,q)$ do not depend on the choice of $\Delta_1,\ldots,\Delta_n$ and for $q\geq 0$, $\tau(\mus,q)=\tau(\nu,q)$.
\end{theorem}
\begin{proof}
    We split the proof into two parts for clarity.
\begin{proofpart}
    The statement (i) holds, the values of $\tau(\nu,q)$ do not depend on the choice of $\Delta_1,\ldots,\Delta_n$, and for $q\geq 0$, $\tau(\mus,q)=\tau(\nu,q)$.
\end{proofpart}
    Let $U_0$ be an open ball which attains the maximal value in \cref{e:wsc-max}.

    To verify (i), by \cref{p:multi-wsc}, it suffices to show that
    \begin{align*}
        \tau_{U_0}(\mus,q)&=\tau(\nu,q) & D_{U_0}(\mus)&=D(\nu) & \dimH K_{U_0}(\mus,\alpha)&=\dimH K(\nu,\alpha).
    \end{align*}

    Let $\sigma$ be such that $S_\sigma(U_0)\subseteq E$, and we see directly from the definitions and \cref{l:u0-sig} that
    \begin{align*}
        \tau_{U_0}(\mus,q)&=\tau_{S_\sigma(U_0)}(\mus,q)\geq\tau(\nu,q)\\
        D_{U_0}(\mus)&=D_{S_\sigma(U_0)}(\mus)\subseteq D(\nu)\\
        \dimH K_{U_0}(\mus,\alpha)&=\dimH K_{S_\sigma(U_0)}(\nu,\alpha) \leq \dimH K(\nu,\alpha).
    \end{align*}
    We now establish the reverse inequalities.

    That $\tau(\mus,q)=\tau_{U_0}(\mus,q)=\tau(\nu,q)$ for $q\geq 0$ is straightforward; see, for example,~\cite[Prop.~3.1]{fl2009}.

    Otherwise, fix $q<0$.
    Since $U_0$ is open and the $\Delta_i$ are essential, there exist net intervals $\Delta_1^*,\ldots,\Delta_n^*$ such that $\vs(\Delta_i)=\vs(\Delta_i^*)$ for each $1\leq i\leq n$ and the $\Delta_i^*$ are pairwise disjoint.
    By \cref{l:er-extension}, there exist compact intervals $F_i\supseteq\Delta_i^*$ such that the $F_i$ are weakly regular, pairwise disjoint, and have $F_i\subset U_0$.
    Set $E^*\coloneqq F_1\cup\cdots\cup F_n$ and let $\nu^*\coloneqq \nu|_{E^*}$.
    Since $E^*$ is weakly regular, it follows that $\tau_{U_0}(\mus,q) \leq \tau(\nu^*,q)$ by \cref{l:br-restriction}.

    It remains to show that $\tau(\nu^*,q)\leq\tau(\nu,q)$ for $q<0$.
    By \cref{l:nbdet}, get similarities $g_i:\Delta_i\cap K\to\Delta_i^*\cap K$ and some $c_1,c_2>0$ such that if $E\subseteq\Delta_i$ is an arbitrary Borel set,
    \begin{equation}\label{e:comp-star}
        c_1\nu^*(g_i(E))\leq \nu(E)\leq c_2\nu^*(g_i(E))
    \end{equation}
    Let each $g_i$ have contraction ratio $\rho_i$.

    Now let $t>0$ be sufficiently small so that $2t\leq\min\{\diam(\Delta_i):1\leq i\leq n\}$ and let $\epsilon_0>0$ be the constant from weak regularity of $E\cap K$.
    Suppose $\{B(x_j,t)\}_{j=1}^m$ is an arbitrary family of disjoint closed balls where $x_j\in E\cap K$.
    For each $j$, there is some $i(j)$ and $y_j$ such that
    \begin{equation}\label{e:yj-map}
        B(y_j,t\epsilon_0/4)\subseteq\Delta_{i(j)}\cap B(x_j,t)
    \end{equation}
    (this must hold for either $y_j=x_j$ or $y_j\in E\cap K\cap B(x_j,t/2)\setminus B(x_j,t\epsilon_0/2)$).

    Now set
    \begin{equation*}
        \rho_0=\frac{\epsilon_0}{4}\min\{\rho_i:1\leq i\leq n\}.
    \end{equation*}
    For each $1\leq j\leq m$, by \cref{e:yj-map},
    \begin{align*}
        \nu(B(x_j,t))&\geq\nu(B(y_j,t\epsilon_0/4) \geq c_1\nu^*(B(g_{i(j)}(y_j),\rho_{i(j)}t\epsilon_0/4))\\
                     &\geq c_1\nu^*(B(g_{i(j)}(y_j),\rho_0 t))
    \end{align*}
    so that $\nu(B(x_j,t))^q\leq c_1^q \nu^*(B(x_j^*,\rho_0 t))^q$ where $x_j^*=g_{i(j)}(y_j)$.
    Observe also that the $B(x_j^*,\rho_0 t)$ are pairwise disjoint.
    But $\{B(x_j,t)\}_{j=1}^m$ was an arbitrary cover, so that
    \begin{equation*}
        \frac{\log \sup\sum_{j}\nu(B(x_j,t))^q}{\log t}\geq \frac{\log c_1^q+\log\sup\sum_{j} \nu^*(B(x_j^*,\rho_0 t))^q}{\log \rho_0^{-1}+\log \rho_0t}.
    \end{equation*}
    Taking limits, it follows that $\tau(\nu,q)\geq\tau(\nu^*,q)$ for $q<0$.

    We now see that $D(\nu)\subseteq D_{U_0}(\mus)$.
    First note that $D_{U_0}(\mus)=[\alpha_{\min},\alpha_{\max}]$ where
    \begin{align*}
        \alpha_{\min} &= \lim_{q\to +\infty}\frac{\tau(\nu,q)}{q}=\lim_{q\to +\infty}\frac{\tau_{U_0}(\mus,q)}{q} & \alpha_{\max}&=\lim_{q\to -\infty}\frac{\tau(\nu,q)}{q}=\lim_{q\to -\infty}\frac{\tau_{U_0}(\mus,q)}{q},
    \end{align*}
    since $\tau(\nu,q)=\tau_{U_0}(\mus,q)$.
    Let $x\in\supp \nu$ be arbitrary with $\alpha=\diml\nu(x)$.
    Then for any $q\in\R$ and $t>0$, we have
    \begin{equation*}
        \log \sup \sum_i \nu(B(x_i,t))^q\geq\log \nu(B(x,t))^q
    \end{equation*}
    where the supremum is over disjoint balls $B(x_i,t)$ with $x_i\in\supp\nu$, and therefore $\tau(\nu,q)\leq q \alpha$.
    Since $\tau(\nu,q)$ is concave, it follows that $\alpha\in[\alpha_{\min},\alpha_{\max}]=D_{U_0}(\mus)$.

    Finally, we verify that $\dimH K(\nu,\alpha)\leq \dimH K_{U_0}(\mus,\alpha)$.
    First note by \cref{e:comp-star} that if $x\in\Delta_i^\circ\cap K$ for some $i$, then $g_i(x)\in(\Delta_i^*)^\circ\cap K\subset U_0$ has
    \begin{equation*}
        \diml \nu(x)=\dim\nu^*(g_i(x))=\diml\mus(g_i(x)).
    \end{equation*}
    Thus $g_i(K(\nu,\alpha)\cap\Delta_i^\circ)\subseteq K_{U_0}(\mus,\alpha)$ and
    \begin{equation*}
        \dimH\Bigl(K(\nu,\alpha)\cap\bigcup_{i=1}^n\Delta_i^\circ\Bigr)\leq\dimH K_{U_0}(\mus,\alpha).
    \end{equation*}
    Since $D(\nu)=D_{U_0}(\mus)$ and $E\setminus\bigcup_{i=1}^n\Delta_i^\circ$ is a finite set (and hence has Hausdorff dimension 0), the result follows.

    Thus the complete multifractal formalism holds.

    Since $U_0$ was fixed, $\tau(\nu,q)$ does not depend on the choice of $\Delta_1,\ldots,\Delta_n$.
\begin{proofpart}
    Statements (ii) and (iii) hold.
\end{proofpart}
    We now see that
    \begin{equation}\label{e:deq}
        D(\nu) = \{\diml\mus(x):x\in K_{\ess},\diml\mus(x)\text{ exists}\}.
    \end{equation}
    If $x\in K_{\ess}$, by \cref{l:er-extension}, there is a weakly regular finite union of essential net intervals $F$ such that $x\in(F\cap K)^\circ$ where we take the interior relative to $K$, and
    \begin{equation*}
        \diml\mus(x)=\diml\mus|_F(x)\in D(\nu)
    \end{equation*}
    since $D(\nu)=D(\mus|_F)$ as proven above.
    Conversely, if $\alpha\in D(\nu)$, then there exists some $y\in U_0$ such that $\diml\mus(y)=\alpha$.
    But $U_0\subseteq K_{\ess}$ by \cref{p:mx-wsc}, so that \cref{e:deq} follows.

    By \cref{t:per-dens}, we have that
    \begin{equation*}
        P(\mus)=\{\diml\mus(x):x\in K_{\ess},x\text{ periodic}\}
    \end{equation*}
    is dense in the set of upper and lower local dimensions in $K_{\ess}$.
    Now $P(\mus)\subseteq D(\nu)$ from \cref{e:deq} and $D(\nu)=[\alpha_{\min},\alpha_{\max}]$ is a closed set with $D(\nu)\subseteq\{\dimul\mus(x):x\in K_{\ess}\}$.
    But again, \cref{t:per-dens} shows that $P(\mus)$ is a dense subset of $\{\dimul\mus(x):x\in K_{\ess}\}$, forcing
    \begin{equation*}
        D(\nu) = \{\dimul\mus(x):x\in K_{\ess}\}.
    \end{equation*}
    Of course, we also have $D(\nu)=\{\dimll\mus(x):x\in K_{\ess}\}$ by the same argument, finishing the proof of the theorem.
\end{proof}
\begin{corollary}\label{c:meas-multi}
    Let $\{S_i\}_{i\in\mathcal{I}}$ be an IFS satisfying the weak separation condition with associated self-similar measure $\mus$.
    Then there exists a sequence of non-empty compact sets $(K_m)_{m=1}^\infty$ with $K_m\subseteq K_{m+1}\subseteq K$ for each $m\in\N$ such that
    \begin{enumerate}[nl,r]
        \item $\lim_{m\to\infty}\mus(K_m)=1$,
        \item each $\mu_m\coloneqq \mus|_{K_m}$ satisfies the complete multifractal formalism, and
        \item $\tau(\mu_m,q)$ and $D(\mu_m)$ do not depend on the index $m$.
    \end{enumerate}
\end{corollary}
\begin{proof}
    Since $\mus$ is Borel and $K_{\ess}$ is a relatively open subset of $K$ with $\mus(K_{\ess})=1$ by \cref{t:ess-meas}, there exists a nested sequence of compact sets $(F_m)_{m=1}^\infty$ with $F_m\subset K_{\ess}$ such that $\lim_{m\to\infty}\mu(F_m)=1$.
    Let $K_m\supseteq F_m$ be a finite union of essential net intervals given by \cref{l:er-extension}.
    Then by \cref{t:multi-wsc}, each $\mu_m\coloneqq \mus|_{K_m}$ satisfies the complete multifractal formalism and $\tau(\mu_m,q)$ and $D(\mu_m)$ do not depend in the index $m$, as required.
\end{proof}
In some situations, the above theorem can also be used to verify that the complete multifractal formalism holds with respect to the invariant measure $\mus$.
\begin{corollary}\label{c:nl-multi}
    Suppose $\{S_i\}_{i\in\mathcal{I}}$ is an IFS satisfying the weak separation condition with transition graph $\mathcal{G}$.
    Suppose there is a bound on the maximum length of a path with no vertices in the essential class.
    Then if $\mus$ is any associated self-similar measure, $\mus$ satisfies the complete multifractal formalism and the local dimensions at periodic points are dense in the set of all local dimensions in $K$.
\end{corollary}
\begin{proof}
    If $M$ is the bound on the maximum length of a path, since $L(e)\geq r_{\min}$ for any $e\in E(\mathcal{G})$, we have that any net interval in $\mathcal{F}_{r_{\min}^M}$ is an essential net interval.
    In particular, $\supp\mus$ is contained in a finite union of essential net intervals, which is automatically boundary regular.
    Apply \cref{t:multi-wsc}.
\end{proof}
\begin{remark}
    For example, if the neighbour set $\vs([0,1])=\{x\mapsto x\}$ is contained in the essential class, then $\mathcal{G}=\mathcal{G}_{\ess}$ and the conditions for the \cref{c:nl-multi} are satisfied.
\end{remark}
\begin{corollary}
    Suppose $\{S_i\}_{i\in\mathcal{I}}$ is an IFS such that the associated transition graph $\mathcal{G}$ is finite.
    Suppose that any cycle in $\mathcal{G}$ is contained in the essential class.
    Then if $\mus$ is any associated self-similar measure, $\mus$ satisfies the complete multifractal formalism and the local dimensions at periodic points are dense in the set of all local dimensions in $K$.
\end{corollary}
\begin{proof}
    When $\mathcal{G}$ is finite, the assumption in \cref{c:nl-multi} is equivalent to the assumption that any cycle is contained in the essential class.
\end{proof}

\section{The finite neighbour condition and examples}\label{s:fn-ex}
\subsection{The finite neighbour condition}
Let $\{S_i\}_{i\in\mathcal{I}}$ be an IFS as in \cref{e:ifs}.
The finite neighbour condition was defined in~\cite{hhr2021} in a way following naturally from the finite type conditions studied in the literature~\cite{ln2007,nw2001}.
\begin{definition}
    We say that $\{S_i\}_{i\in\mathcal{I}}$ satisfies the \defn{finite neighbour condition} if there are only finitely many neighbour sets.
    Equivalently, its transition graph $\mathcal{G}$ is finite.
\end{definition}
\begin{remark}\label{r:fnc-d}
    The definition of a neighbour in \cref{d:nb} differs slightly from \cite[Def'n.~2.7]{hhr2021}.
    Namely, for a net interval $\Delta\in\mathcal{F}$ and $T\in\vs(\Delta)$, we require $T(K)\cap(0,1)\neq\emptyset$ rather than $T([0,1])\supseteq[0,1]$.
    However, using \cite[Cor.~3.4]{dln2013} with respect to the generation $k_0\coloneqq  r_{\min}/M$ where $M=\sup_{\Delta\in\mathcal{F}}\lm(\Delta)$ and the characterization \cite[Thm~3.4.]{hhr2021}, one can verify that the finiteness assumptions are in fact equivalent.
\end{remark}
It is shown in~\cite{hhr2021} that the finite neighbour condition is equivalent to the generalized finite type condition~\cite{ln2007} holding with respect to the invariant open set $(0,1)$.
Moreover, under the assumption that the attractor $K$ is an interval, it is proven in~\cite{fen2016,hhr2021} that the finite neighbour condition is in fact equivalent to the weak separation condition.
The author is not aware of any IFS of similarities in $\R$ which satisfies the weak separation condition but not the finite neighbour condition.

Of course, when an IFS satisfies the finite neighbour condition, it also satisfies the weak separation condition (see, for example, \cite[Thm.~1.1]{ln2007} or \cite[Thm.~3.7]{hhr2021}) and thus has a unique finite essential class $\mathcal{G}_{\ess}$.
Interestingly, the converse also holds:
\begin{theorem}\label{t:fnc-es}
    The IFS $\{S_i\}_{i\in\mathcal{I}}$ satisfies the finite neighbour condition if and only if $\mathcal{G}(\{S_i\}_{i\in\mathcal{I}})$ has a finite essential class.
\end{theorem}
\begin{proof}
    \impr
    Since the finite neighbour condition implies the weak separation condition, this follows immediately from \cref{p:wsc-es} since $\mathcal{G}$ is a finite graph.

    \impl
    We first define a construction on neighbour sets.
    Let $v_1=\{f_1,\ldots,f_{\ell_1}\}$ and $v_2=\{g_1,\ldots,g_{\ell_2}\}$ be a pair of neighbour sets.
    We denote by $J(v_1,v_2)$ the set of all subsets $w=\{h_1,\ldots,h_m\}$ such that there exist indices $i,j$ and $T=f_i\circ g_j^{-1}$ such that
    \begin{equation*}
        \{T_\Delta\circ h_1,\ldots,T_\Delta\circ h_m\}\subset \{f_1,\ldots,f_{\ell_1}\}
    \end{equation*}
    where $\Delta=\bigl[\min\{0,T(0),T(1)\},\max\{1,T(0),T(1)\}\bigr]$ and $T_\Delta(x)=rx+d$ with $r>0$ where $T_\Delta([0,1])=\Delta$.
    Clearly there are only finitely many functions $T$, so that $J(v_1,v_2)$ is a finite set.
    When $F$ is a finite set, we denote by $J(F)=\bigcup_{v_1,v_2\in F}J(v_1,v_2)$, which is also finite.

    Now, by assumption, $\mathcal{G}$ has a finite essential class $\mathcal{G}_{\ess}$ so that $J_0\coloneqq J(V(\mathcal{G}_{\ess}))$ is finite.
    Let $\Delta_0\in\mathcal{F}_\alpha$ be an arbitrary net interval; we will see that $\vs(\Delta_0)\in J_0$, from which it follows that $\{S_i\}_{i\in\mathcal{I}}$ satisfies the finite neighbour condition.

    First, let $\sigma$ be such that $r_\sigma>0$ and $S_\sigma([0,1])$ is a finite union of essential net intervals (just take $\sigma$ such that $S_\sigma([0,1])$ is contained in some essential net interval; if $r_\sigma<0$, append some $i\in\mathcal{I}$ with $r_i<0$).
    Let $\vs(\Delta_0)$ have neighbours generated by words $\{\omega_1,\ldots,\omega_m\}$ in $\Lambda_\alpha$; note that each $\sigma\omega_i\in\Lambda_{r_\sigma\alpha}$.
    Let $\Delta_1=S_\sigma(\Delta_0)$ and write $\Delta_1=[a,b]$.
    Then there exist essential net intervals $\Delta_a,\Delta_b\in\mathcal{F}_{r_\sigma\alpha}$ such that $\Delta_a=[a,a_0]$ and $\Delta_b=[b_0,b]$; perhaps $\Delta_a=\Delta_b$.
    Note that $\sigma\omega_1$ has $\Delta_a,\Delta_b\subseteq S_{\sigma\omega_1}([0,1])$ since $\Delta_a,\Delta_b\subseteq\Delta_1$ so that $\sigma\omega_1$ generates a neighbour $f_a$ of $\Delta_a$ and $f_b$ of $\Delta_b$.

    We see that $\vs(\Delta_0)$ is a join of $(\vs(\Delta_a),\vs(\Delta_b))$.
    Set $T=f_a\circ f_b^{-1}$.
    We first note that
    \begin{itemize}[nl]
        \item $T_{\Delta_a}\circ f_a=S_{\sigma\omega_1}=T_{\Delta_b}\circ f_b$, so that $T\coloneqq f_a\circ f_b^{-1}=T_{\Delta_b}\circ T_{\Delta_a}^{-1}$ and
        \item $\Delta\coloneqq \bigl[\min\{0,T(0),T(1)\},\max\{1,T(0),T(1)\}\bigr]=T_{\Delta_a}^{-1}(\Delta_1)$ so that $T_\Delta=T_{\Delta_a}^{-1}\circ T_{\Delta_1}$.
    \end{itemize}
    Now let $h\in\mathcal{V}(\Delta_0)$ be arbitrary.
    Since $r_\sigma>0$, $T_{\Delta_1}=S_\sigma\circ T_{\Delta_0}$.
    Then if $h=T_{\Delta_0}^{-1}\circ S_{\omega_i}$, we have
    \begin{equation*}
        T_\Delta\circ h=(T_{\Delta_a}^{-1}\circ T_{\Delta_1})\circ(T_{\Delta_1}^{-1}\circ S_\sigma\circ S_{\omega_i})=T_{\Delta_a}^{-1}\circ S_{\sigma\omega_i},
    \end{equation*}
    where $\sigma\omega_i$ generates a neighbour of $\Delta_a$, and thus $T_\Delta\circ h\in\vs(\Delta_a)$, as required.
\end{proof}
\begin{remark}
    If $v,w\in V(\mathcal{G}_{\ess})$, then there are at most $\# v\cdot \# w$ distinct functions $T$, so that $\# J(v,w)\leq \# v\cdot\# w\cdot 2^{\# v}$.
    Moreover, there are at most $(\# V(\mathcal{G}_{\ess}))^2$ pairs $(v,w)$.
    In particular, if there are $m$ distinct neighbours in $\mathcal{G}_{\ess}$, then $\# V(\mathcal{G}_{\ess})\leq 2^m$ and $\#v\leq m$ for any $v\in V(\mathcal{G}_{\ess})$, so that
    \begin{equation*}
        \# V(\mathcal{G})\leq (\# V(\mathcal{G}_{\ess}))^2\cdot m^2\cdot 2^m\leq m^2 8^{m}.
    \end{equation*}
    Thus the above proof gives a quantitative bound on the size of $\mathcal{G}$ as a function of the number of distinct neighbours in $\mathcal{G}_{\ess}$.
\end{remark}
\subsection{Approximate transition matrices}
Under the finite neighbour condition, we may approximate the transition matrix $T(e)$ by the matrix $T^*(e)$ given by $T^*(e)_{ij}=p_\ell$ in the same context as \cref{e:trmat}.
Since there are only finitely many values $\frac{\mus(f_i^{-1}((0,1))}{\mus(g_j^{-1}((0,1))}$, there exist constants $c_1,c_2>0$ such that $c_1T^*(\eta)\leq T(\eta)\leq c_2T^*(\eta)$ element-wise for any admissible path $\eta$.
Moreover, since $\mus$ is a probability measure, direct computation shows that $\norm{T^*(\eta)}_1\leq\mus(\Delta)$.
Applying \cref{t:approx}, we have:
\begin{corollary}\label{c:fapprox}
    Let $\{S_i\}_{i\in\mathcal{I}}$ be an IFS satisfying the finite neighbour condition with associated self-similar measure $\mus$.
    \begin{itemize}[nl]
        \item There exist constants $c_1,c_2>0$ such that for any path $\eta$ realized by $(\Delta_i)_{i=0}^n$,
            \begin{equation*}
                c_1\muv(\Delta_m)\preccurlyeq T^*(\eta)\muv(\Delta_0)\preccurlyeq c_2\muv(\Delta_m)
            \end{equation*}
            where the inequalities hold pointwise.
        \item There exists a constant $c>0$ such that for any $\Delta\in\mathcal{F}$ with symbolic representation $\eta$,
            \begin{equation*}
                c\mus(\Delta)\leq \norm{T^*(\eta)}_1\leq \mus(\Delta).
            \end{equation*}
    \end{itemize}
\end{corollary}
One may also observe that the same principle works for periodic points.
We have the natural analogue of \cref{p:period}:
\begin{corollary}\label{c:fperiod}
    Let $\{S_i\}_{i\in\mathcal{I}}$ be any IFS and suppose $x$ is a periodic point with period $\theta=(e_1,\ldots,e_s)$.
    Then the local dimension of $\mu$ at $x$ exists and is given by
    \begin{equation*}
        \dim_{loc}\mu(x)=\frac{\log\spr(T^*(\theta))}{\log L(\theta)}
    \end{equation*}
    where if $x$ is a boundary point of a net interval with two different symbolic representations given by periods $\theta$ and $\phi$, then $\theta$ is chosen to satisfy
    \begin{equation*}
        \frac{\log\spr(T^*(\theta))}{\log L(\theta)}\geq \frac{\log\spr(T^*(\phi))}{\log L(\phi)}.
    \end{equation*}
\end{corollary}
\begin{proof}
    The proof is identical to the proof of \cref{p:period}, noting that the analogue of \cref{c:fapprox} holds since the set $\{\vs(\Delta):x\in\Delta,\Delta\in\mathcal{F}\}$ is finite.
\end{proof}

\subsection{An overlapping IFS with non-commensurable contraction ratios}\label{ss:ifs-noncomm}
Consider the IFS given by the maps
\begin{align*}
    S_1(x) &=\rho\cdot x & S_2(x) &= r\cdot x+\rho(1-r) & S_3(x) &=r\cdot x+1-r
\end{align*}
where $0<\rho,r<1$ satisfy $\rho+2r-\rho r \leq 1$, i.e. $S_2(1) \leq S_3(0)$.
This IFS was initially studied by \cite{lw2004} and was the first example of an iterated function system with overlaps and satisfying the weak separation condition without commensurable contraction ratios.
It is known that the Hausdorff dimension of the attractor $K$ is the unique solution to the equation $\rho^s+2r^s-(\rho r)^s=1$ (see \cite[Prop.~4.9]{lw2004} or \cite[Ex.~5.1]{ln2007}).

Under the assumption that $\rho>r>\rho^2$, we will compute the neighbour sets and the transition graph.
We also give formulas to compute the range of local dimensions.
We will also show (for all valid parameters $r,\rho$) that any associated self-similar measure satisfies the complete multifractal formalism.

\subsubsection{Neighbour sets and the transition graph}\label{sss:ns-tg}
We first compute the neighbour sets and children in complete detail.
The net interval $\Delta_0=[0,1]$ has $\vs(\Delta_0)=\{x\mapsto x\}$ and $\tg(\Delta_0)=1=m(\Delta_0)\cdot 1$ since 1 is the maximal contraction ratio of any of its neighbours.
Thus $\Delta_0$ has children
\begin{equation*}
    (\Delta_1=[0,\rho(1-r)],\Delta_2=[\rho(1-r),\rho],\Delta_3=[\rho,\rho+r-\rho r],\Delta_4=[1-r,r])
\end{equation*}
in $\mathcal{F}_1$.
Note that when $\rho+2r-\rho r<1$, $[\rho+r-\rho r,1-r]$ is not a net interval since its interior does not intersect $K$.
One may compute
\begin{align*}
    \vs(\Delta_1) &= \{x\mapsto x/(1-r)\} & \vs(\Delta_2) &= \{x\mapsto x/\rho,x\mapsto x/r+\frac{1}{r}-1\}\\
    \vs(\Delta_3) &= \{x\mapsto\frac{x}{1-\rho}+\frac{\rho}{1-\rho}\} & \vs(\Delta_4) &= \{x\mapsto x\}.
\end{align*}
Since $\vs(\Delta_4)=\vs(\Delta_0)$, the children of $\Delta_4$ are scaled versions of the children of $\Delta_0$ and have the same neighbour sets by \cref{t:ttype}.
\begin{itemize}[nl]
    \item Since $\rho>r$, $\Delta_1$ has $\tg(\Delta_1)=m(\Delta_1)\cdot (1/(1-r))=\rho$, so $\Delta_1$ has children
        \begin{equation*}
            (\Delta_5 = [0,\rho^2(1-r)],\Delta_6 = [\rho^2(1-r),\rho^2],\Delta_7=[\rho^2,\rho(\rho+r-\rho r)])
        \end{equation*}
        where $\vs(\Delta_5)=\vs(\Delta_1)$, $\vs(\Delta_6)=\vs(\Delta_2)$, and $\vs(\Delta_7)=\vs(\Delta_3)$.
    \item $\Delta_2$ has $\tg(\Delta_2)=\rho$ and one child $\Delta_8=[\rho-\rho r,\rho]$ with $\vs(\Delta_8)=\{x\mapsto x,x\mapsto x/\rho\}$.
        Note that $\Delta_8=\Delta_2$, but $\vs(\Delta_8)\neq\vs(\Delta_2)$.
    \item $\Delta_3$ has $\tg(\Delta_3)=r$ and two children $(\Delta_9=[\rho,\rho+r^2-\rho r^2],\Delta_{10}=[r-r^2,r])$ with $\vs(\Delta_9)=\vs(\Delta_3)$ and $\vs(\Delta_{10})=\vs(\Delta_0)$.
    \item $\Delta_8$ has children $\Delta_{11}=[\rho-\rho r,\rho-\rho r^2],\Delta_{12}=[\rho-\rho r^2,\rho]$ with $\vs(\Delta_{11})=\vs(\Delta_1)$ and $\vs(\Delta_{12})=\vs(\Delta_2)$.
\end{itemize}
Thus by \cref{t:ttype}, there are no new neighbour sets and the IFS satisfies the finite neighbour condition.

For simplicity, fix $v_0=\vs(\Delta_0)$, $v_1=\vs(\Delta_1)$, $v_2=\vs(\Delta_2)$, $v_3=\vs(\Delta_3)$ and $v_4=\vs(\Delta_8)$.
Let $\mu_{\bm{p}}$ be a self-similar measure associated with the IFS, where $\bm{p}=(p_1,p_2,p_3)$.
Observing that $v_4$ has exactly one child, we can construct an equivalent transition graph by removing $v_4$, concatenating the incoming edges with the outgoing edge, and multiplying the corresponding edge lengths and transition matrices.
This results in the modified transition graphs and edge lengths described in \cref{f:tgraph-m}.
\begin{figure}[ht]
    \begin{tikzpicture}[
        baseline=(current bounding box.center),
        vtx/.style={circle,inner sep=1pt,draw=black,fill=red!30},
        elbl/.style={fill=white,circle,inner sep=1pt},
        edge/.style={thick,->,>=stealth}
        ]
        \node[vtx] (v0) at (0,5) {$v_0$};
        \node[vtx] (v1) at (4,0) {$v_1$};
        \node[vtx] (v3) at (4,5) {$v_3$};
        \node[vtx] (v4) at (0,0) {$v_4$};

        \draw[edge] (v0) -- node[elbl]{$e_0$}(v1);
        \draw[edge] (v0) -- node[elbl]{$e_1'$}(v4);
        \draw[edge] (v0) .. controls +(90:2) and +(180:2) .. node[elbl]{$e_3$} (v0);

        \draw[edge] (v1) .. controls +(0:2) and +(270:2) .. node[elbl]{$e_4$} (v1);
        \draw[edge] (v1) -- node[elbl]{$e_6$} (v3);


        \draw[edge] (v3) .. controls +(0:2) and +(90:2) .. node[elbl]{$e_8$} (v3);
        \draw[edge] (v3) to[out=180-12,in=12] node[elbl]{$e_9$} (v0);

        \draw[edge] (v4) to[out=-12,in=180+12]  node[elbl]{$e_{10}$} (v1);
        \draw[edge] (v1) to[out=180-12,in=12] node[elbl]{$e_5'$} (v4);
        \draw[edge] (v4) .. controls +(180:2) and +(270:2) .. node[elbl]{$e_{11}'$} (v4); 
        \draw[edge] (v0) to[out=-12,in=180+12] node[elbl]{$e_2$} (v3);
    \end{tikzpicture}
    \begin{tabular}{ccc}
        Edge & Length & Transition Matrix\\
        $e_0$ & $\rho$ & $\begin{pmatrix}p_1\end{pmatrix}$\\
        $e_1'$ & $r$ & $\begin{pmatrix}p_1p_3&p_2\end{pmatrix}$\\
        $e_2$ & $r$ & $\begin{pmatrix}p_2\end{pmatrix}$\\
        $e_3$ & $r$ & $\begin{pmatrix}p_3\end{pmatrix}$\\
        $e_4$ & $\rho$ & $\begin{pmatrix}p_1\end{pmatrix}$\\
        $e_5'$ & $\rho$ & $\begin{pmatrix}p_1p_3&p_2\end{pmatrix}$\\
        $e_6$ & $r$ & $\begin{pmatrix}p_2\end{pmatrix}$\\
        $e_8$ & $r$ & $\begin{pmatrix}p_2\end{pmatrix}$\\
        $e_9$ & $r$ & $\begin{pmatrix}p_3\end{pmatrix}$\\
        $e_{10}$ & $r$ &$\begin{pmatrix}1\\p_1\end{pmatrix}$\\
        $e_{11}'$ & $r$ & $\begin{pmatrix}p_3&0\\p_1p_3&p_2\end{pmatrix}$
    \end{tabular}
    \caption{Modified transition graph with edge lengths and transition matrices}
    \label{f:tgraph-m}
\end{figure}
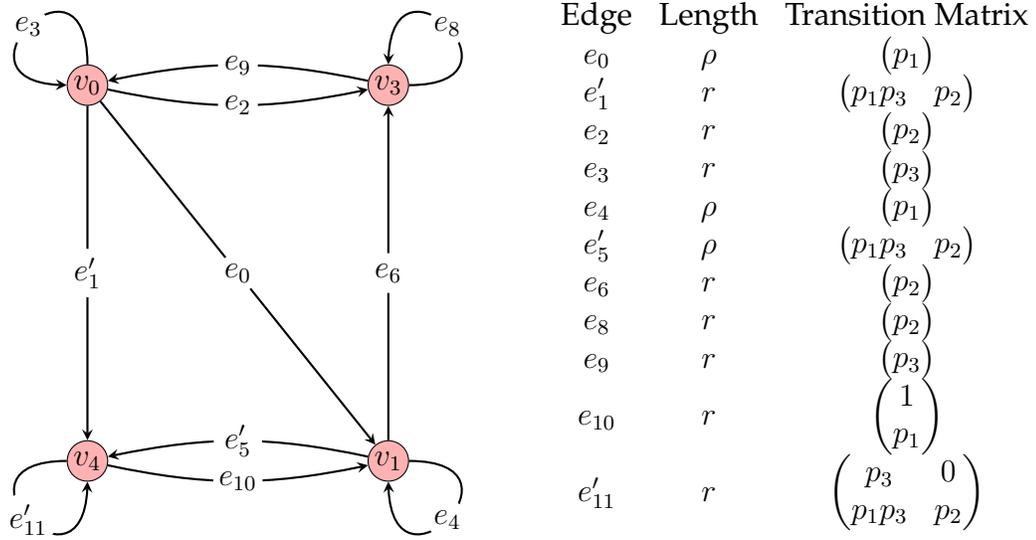
\subsubsection{The attainable local dimensions}\label{sss:non-comm-dim}
We see that the conditions for \cref{c:nl-multi} are satisfied, so that the measure $\mus$ satisfies the complete multifractal formalism and that the local dimensions at periodic points are dense in the set of upper and lower local dimensions.

We now compute the range of local dimensions at periodic points.
We first make note of the following obvious inequality: if $0<a,b,c,d$ and $\log a/\log b\leq\log c/\log d$, then
\begin{equation}\label{e:ineq}
    \frac{\log a}{\log b}\leq\frac{\log ac}{\log bd}\leq\frac{\log c}{\log d}.
\end{equation}

Now let $\eta$ be any cycle contained in $\mathcal{G}$.
If $\eta$ only passes through $v_4$, since $\spr{T^*(e_{11}')}=\max\{p_2,p_3\}$, the local dimension corresponding to the cycle $(e_{11}')$ is $\frac{\log \max\{p_2,p_3\}}{\log r}$.
Otherwise, $\eta$ passes through some vertex other than $v_4$.
Thus without loss of generality, $\eta$ begins and ends and some vertex $v\neq v_4$.
Suppose $\eta$ visits some vertex $w\neq v_4$ twice, i.e. $\eta=\eta_1\eta_2\eta_3$ where $\eta_1$ is a path from $v$ to $w$, $\eta_2$ is a cycle from $w$ to $w$, and $\eta_3$ is a path from $w$ to $v$.
Then $\eta$ can be written as a concatenation of cycles $\eta_2$ and $\eta_3\eta_1$, where $T(\eta_2)$ and $T(\eta_3\eta_1)$ are singletons, and by \cref{e:ineq}, we have that
\begin{align*}
    \min\Bigl\{\frac{\log\spr T(\eta_2)}{\log L(\eta_2)},\frac{\log\spr T(\eta_3\eta_1)}{\log L(\eta_3\eta_1)}\Bigr\}&\leq \frac{\log\spr T(\eta)}{\log L(\eta)}\\
                                                                                                                     &\leq\max\Bigl\{\frac{\log\spr T(\eta_2)}{\log L(\eta_2)},\frac{\log\spr T(\eta_3\eta_1)}{\log L(\eta_3\eta_1)}\Bigr\}.
\end{align*}
In other words, the minimum and maximum local dimensions on cycles are attained at cycles which do not repeat any vertex other than $v_4$.
Thus it suffices to consider all such families of cycles.

If $\eta$ does not pass through $v_4$, the only non-repeating cycles are $(e_3)$, $(e_4)$, $(e_8)$, and $(e_2,e_9)$.
We thus see that the maximum and minimum possible local dimensions are attained at the points in
\begin{equation*}
    S = \bigl\{\frac{\log p_1}{\log \rho},\frac{\log p_2}{\log r},\frac{\log p_3}{\log r}\bigr\}.
\end{equation*}

Otherwise, $\eta$ passes through $v_4$.
A straightforward induction argument shows that
\begin{equation*}
    T^*(e_{11}')^n =
    \begin{cases}
        \begin{pmatrix}p_3^n & 0\\\frac{p_1p_3(p_2^n-p_3^n)}{p_2-p_3}&p_2^n\end{pmatrix} &: p_2\neq p_3\\
        \begin{pmatrix}p^n&0\\np^np_1 & p^n\end{pmatrix} &: p_2=p_3=:p
    \end{cases}
    .
\end{equation*}
Now, let
\begin{align*}
    \eta_{1,n}&=(e_6,e_9,e_1',\underbrace{e_{11}',\ldots,e_{11}'}_n,e_{10}) & \eta_{2,n}&=(e_5',\underbrace{e_{11}',\ldots,e_{11}'}_n,e_{10})
\end{align*}
denote the two possible families of cycles which go through $v_4$ and do not repeat a vertex not in $v_4$.
We then have that
\begin{align*}
    a_n\coloneqq \spr T^*(\eta_{1,n})&=
    \begin{cases}
        \frac{p_1p_2p_3(p_2^{n+2}-p_3^{n+2})}{p_2-p_3} &: p_2\neq p_3\\
        (2+n)p^{n+2}(1-2p)^2 &: p_2=p_3=:p
    \end{cases}
    \\
    b_n\coloneqq \spr T^*(\eta_{2,n})&=
    \begin{cases}
        \frac{p_1(p_2^{n+2}-p_3^{n+2})}{p_2-p_3} &: p_2\neq p_3\\
        (2+n)p^{n+1}(1-2p)^2 &: p_2=p_3=:p
    \end{cases}
    \\
    L(\eta_{1,n})&=r^{n+4}\\
    L(\eta_{2,n})&= \rho r^{n+1}.
\end{align*}
Let
\begin{align*}
    a_{\min} &= \inf_n\frac{\log a_n}{(n+4)\log r} & a_{\max} &= \sup_n\frac{\log a_n}{(n+4)\log r}\\
    b_{\min} &= \inf_n\frac{\log b_n}{(n+1)\log r+\log\rho} & b_{\max} &= \sup_n\frac{\log b_n}{(n+1)\log r+\log\rho}.
\end{align*}
Then the minimal local dimension is equal to
\begin{equation*}
    \alpha_{\min}\coloneqq \min\bigl\{\frac{\log p_1}{\log \rho},\frac{\log p_2}{\log r},\frac{\log p_3}{\log r},a_{\min},b_{\min}\bigr\}.
\end{equation*}
and the maximal local dimension is equal to
\begin{equation*}
    \alpha_{\max}\coloneqq \max\bigl\{\frac{\log p_1}{\log \rho},\frac{\log p_2}{\log r},\frac{\log p_3}{\log r},a_{\max},b_{\max}\bigr\}.
\end{equation*}
The parameters $\alpha_{\min}$ and $\alpha_{\max}$ can be determined exactly in many situations, but generic solutions are tedious.
Additional details are left to the reader.

\subsubsection{The maximal open sets of the weak separation condition}\label{sss:valid-open}
Here we show, under the same assumption $\rho>r>\rho^2$ that the essential net interval $[0,1]\cap K$ is not contained in a union of open balls $U_0$ satisfying the maximal value in \cref{e:wsc-max}.
In fact, we show that for any $\epsilon>0$, the open set $(1-\epsilon,1)\cap K$ is not contained a finite union of such open balls.
In addition, this shows that for any $U(x,t)$ with $\#\mathcal{S}_t(U(x,t))$ maximal, we must have $1\notin U(x,t)$, whereas $1\in K=K_{\ess}$.
A similar argument gives this result for general parameters $\rho$ and $r$, but the details are tedious and we omit the proof.

We first note that $\sup_{x\in\R,t>0}\#\mathcal{S}_t(U(x,t))\geq 5$.
To see this, take $t=1/4$ and $U_0\coloneqq U(1/4,1/4)$.
Then for each $\sigma\in\{11,12,13,22,23\}$, we have $S_\sigma(K)\cap U_0\neq\emptyset$ (since $S_{13}=S_{21}$, we exclude the word $21$).

To show that $(1-\epsilon,1)\cap K$ is not contained in a finite union of maximal open balls for each $\epsilon>0$, since $1$ is an accumulation point for $K$ it suffices to show that if $t>0$ and $U(x,t)$ is any open ball such that $x+t=1$, $\#\mathcal{S}_t(U(x,t))<5$.
A direct check shows that for $t>1/4$, $\#\mathcal{S}_t(U(x,t))<5$.
Otherwise, let $m\geq 1$ be such that $1/4^{m+1}<t\leq 1/4^m$.
Since the rightmost child of $[0,1]$ is the net interval $[3/4,1]\in\Lambda_{1/4}$ with $\vs([3/4,1])=\vs([0,1])$, the net interval in generation $t$ containing $1$ is the interval $\Delta=[1-1/4^m,1]$ which has $\vs(\Delta)=\vs([0,1])$, and thus $U(x,t)\subseteq\Delta'=[1-1/4^{m-1}]$ where $\vs(\Delta')=\vs([0,1])$.
But then up to normalization, we know that the net intervals contained in $\Delta'$ are the same as the net intervals contained in $[0,1]$ so the case for general $t$ reduces to the case $t>1/4$.

\subsubsection{On the multifractal formalism}
We now dispense with the assumptions on the parameters $\rho,r$ and establish the following result.
\begin{theorem}\label{t:rr-multif}
    Any invariant measure $\mus$ associated with the IFS
    \begin{align*}
        S_1(x) &=\rho\cdot x & S_2(x) &= r\cdot x+\rho(1-r) & S_3(x) &=r\cdot x+1-r
    \end{align*}
    where $0<\rho,r<1$ satisfy $\rho+2r-\rho r \leq 1$ satisfies the complete multifractal formalism.
\end{theorem}
\begin{proof}
    By \cref{c:nl-multi} and the following remark, since the IFS satisfies the weak separation condition, it suffices to show that the vertex $v_0\coloneqq \{x\mapsto x\}$ is contained in the essential class.
    As argued in \cref{sss:ns-tg}, the net interval $[0,1]$ has children
    \begin{equation*}
        \bigl(\Delta_1=[0,\rho(1-r)],\Delta_2=[\rho(1-r),\rho],\Delta_3=[\rho,\rho+r-\rho r],\Delta_4=[1-r,r]\bigr)
    \end{equation*}
    in $\mathcal{F}_1$ with neighbour sets
    \begin{align*}
        \vs(\Delta_1) &= \{x\mapsto x/(1-r)\} & \vs(\Delta_2) &= \{x\mapsto x/\rho,x\mapsto x/r+\frac{1}{r}-1\}\\
        \vs(\Delta_3) &= \{x\mapsto\frac{x}{1-\rho}+\frac{\rho}{1-\rho}\} & \vs(\Delta_4) &= \{x\mapsto x\}.
    \end{align*}
    In particular, there is an edge from $v_0$ to $v_0$.
    Moreover, since the word $23$ is in $\Lambda_r$, where $S_{23}([0,1])$ is disjoint from $S_3([0,1])$, $S_{22}([0,1])$, and $S_1([0,1])$ by the assumptions on $\rho$ and $r$, it follows that $S_{23}([0,1])$ is a net interval with neighbour set $v_0$.
    Thus there is an edge from $\vs(\Delta_3)$ to $v_0$.
    Similarly, the words $11$ and $12$ are in $\Lambda_\rho$, where $S_{12}([0,1])$ is disjoint from $S_2([0,1])$, so as computed in \cref{sss:ns-tg}, the children of $\Delta_1$ have neighbour sets $\vs(\Delta_1)$, $\vs(\Delta_2)$, and $\vs(\Delta_3)$.
    Since there is an edge from $\vs(\Delta_3)$ to $v_0$, there is a path from $\vs(\Delta_1)$ to $v_0$.

    It remains to consider the offspring of $v_2\coloneqq \vs(\Delta_2)$.
    We will treat the case where $r>\rho$; the case where $r\leq\rho$ follows by an analogous argument.
    Let $m$ be maximal such that $r^m>\rho$.
    We will compute the net intervals in generation $\Lambda_{r^m}$.

    For $0\leq k\leq m$ write
    \begin{align*}
        \sigma_k &= \underbrace{2\ldots 2}_{k\text{ times}}1 &\tau_k &= \underbrace{2\ldots 2}_{k\text{ times}}.
    \end{align*}
    For simplicity, given $t>0$, write $\Gamma_t=\{S_\omega:\omega\in\Lambda_t,S_\omega((0,1))\cap\Delta_2\neq\emptyset\}$.
    Note that $S_2(S_1(1))=S_1(1)$ where $S_1(1)$ is the right endpoint of $\Delta_2$, so that $S_{\sigma_k}(S_1(1))=S_1(1)$.
    Thus by choice of $m$, we have for $k\leq m$
    \begin{equation*}
        \Gamma_{r^k}=\{S_{\tau_{k+1}},S_{\sigma_0},S_{\sigma_1},\ldots,S_{\sigma_k}\}.
    \end{equation*}

    First assume $r^{m+1}<\rho$.
    Since $r^{m+1}<\rho$ and $S_1([0,1])\supseteq\Delta_2\supseteq\Delta^{i}$, $\tg(\Delta^{i})=\rho$.
    Thus since $S_{12}(1)\leq S_2(0)$ and $S_{\sigma_0 2}=S_{\sigma_1}$, we have
    \begin{equation*}
        \Gamma_\rho=\{\tau_{m+1},S_{\sigma_1},\ldots,S_{\sigma_m}\}.
    \end{equation*}
    Since $S_\tau(0)>S_{\sigma_m}(1)$, the net intervals in $\mathcal{F}_{\rho}$ contained in $\Delta_2$ are given, ordered from left to right,
    \begin{align*}
        \Delta^{i} &=[S_{\sigma_i}(0),S_{\sigma_{i+1}}(0)] & \Delta^{m} &= [S_{\sigma_m}(0),S_\tau(0)] & \Delta^{m+1} &= [S_\tau(0),S_{\sigma_m}(1)]
    \end{align*}
    for $1\leq i<m$.
    Since $\Gamma_{r^{m-1}}=\{S_2^{-1}\circ g:g\in\Gamma_\rho\}$, for each $1\leq i\leq m+1$, $S_2^{-1}(\Delta^{i})\in\mathcal{F}_{r^{m-1}}$ with $\vs(S_2^{-1}(\Delta^{i}))=\vs(\Delta^{i})$.
    But again $S_2$ fixes the right endpoint of $\Delta_2$, so that
    \begin{equation*}
        \Delta_2\supseteq S_2^{-1}(\Delta^{i})\supseteq\Delta^{j}
    \end{equation*}
    for each $2\leq i\leq m+1$.
    In particular, every child of $\vs(\Delta^{i})$ is of the form $\vs(\Delta^{j})$, and there is a path from $\vs(\Delta^{i})$ to $\vs(\Delta^{1})$ for each $i\geq 2$.
    Moreover, a direct computation shows that $\vs(\Delta^{1})=\vs(\Delta_1)$, so there is a path from $\vs(\Delta^{1})$ to $v_0$.
    Thus there are no new net intervals, and there is a path to $v_0=\{x\mapsto x\}$ from any vertex in the transition graph, as required.

    In the case $r^{m+1}=\rho$, we get
    \begin{equation*}
        \Gamma_\rho=\{\tau_{m+2},S_{\sigma_1},\ldots,S_{\sigma_{m+1}}\}
    \end{equation*}
    so that $\Gamma_\rho$ is a rescaled version of $\Gamma_{r^m}$, and the argument follows similarly.
\end{proof}

\subsection{On an example of Deng and Ngai}
In \cite[Ex.~8.5]{dn2017}, Deng and Ngai introduced the following IFS similar in structure to \cref{ss:ifs-noncomm} but with an additional overlap.
Consider IFS defined by following four maps
\begin{align*}
    S_1(x) &= \rho x & S_2(x) &= rx+\rho(1-r)\\
    S_3(x) &= \rho^{-1}r^2 x+(1-r)(\rho+r) & S_4(x) &= rx+(1-r)
\end{align*}
where $0<\rho,r\in<1$ satisfy $r^2<\rho$ and $\rho(r-1)(\rho+r-1)>r^2$.
The constraints on $\rho$ and $r$ ensure that $S_3((0,1))\cap S_4((0,1))=\emptyset$.

The parameters of this IFS are chosen so that $S_{14}=S_{21}$ and $S_{24}=S_{31}$.
One can verify, arguing similarly to \cref{t:rr-multif}, that $\mathcal{G}=\mathcal{G}_{\ess}$ and hence any associated self-similar measure satisfies the complete multifractal formalism.

\subsection{A modified multifractal formalism for Cantor-like measures}
Consider the family of IFS given by maps
\begin{equation*}
    \bigl\{S_j(x)=\frac{x}{r}+\frac{j}{mr}(r-1):0\leq j\leq m\bigr\}
\end{equation*}
where $m\geq r\geq 2$ and $m,r$ are integers.
This family of IFS, with appropriate probabilities, contains rescaled versions of measures such as convolutions of the usual Cantor measure.
In particular, certain self-similar measures in this family were among the first recognized for which the multifractal formalism can fail~\cite{hl2001}.
The set of local dimensions is known to consist of a closed interval and, with appropriate probabilities, an isolated point.
The $L^q$-spectra have also been computed, as well as a modified multifractal formalism \cite{fl2009,flw2005,lw2005,shm2005}; our results here are minor improvements of existing results and are primarily useful as illustrations of the theorems.

Fix any IFS $\{S_i\}_{i\in\mathcal{I}}$ in this family with attractor $K$ and associated self-similar measure $\mu$.
Arguing similarly to \cite[Prop.~7.1]{hhm2016}, one may verify that $K=[0,1]$, $K_{\ess}=(0,1)$, and
\begin{equation*}
    K_m\coloneqq \bigcup_{\substack{\Delta\in\mathcal{F}_{r^{m-1}}\\\vs(\Delta)\in V(\mathcal{G}_{\ess})}}=\bigl[\frac{r-1}{kr^m},1-\frac{(r-1)}{kr^m}\bigr].
\end{equation*}
Then \cref{t:multi-wsc} (weak regularity is always satisfied since $K=[0,1])$ gives that each $\mu_m\coloneqq \mus|_{K_m}$ satisfies the complete multifractal formalism and
\begin{equation*}
    D(\mu_m)=\{\diml\mus(x):x\in(0,1)\}.
\end{equation*}
This provides an alternative proof of some of the results contained in \cite{flw2005,shm2005} (without the assumption $k<2r-2$) and a variation of~\cite[Ex.~6.2]{fl2009}.

From the perspective of \cref{c:nl-multi}, the obstruction to the multifractal formalism is combinatorial: there is a cycle outside the essential class which contributes a point with local dimension not contained in the closed interval $\{\diml\mu(x):x\in K_{\ess}\}$.

\bibliographystyle{amsplain}
\bibliography{texproject/citations/local-main}
\end{document}